\RequirePackage{fixltx2e}
\documentclass[12pt,twoside,english]{elsarticle}
\usepackage[T1]{fontenc}
\usepackage{color}
\usepackage{mathtools}
\usepackage{amsmath}
\usepackage{amsthm}
\usepackage{amssymb}
\usepackage{graphicx}

\usepackage{xcolor}
\usepackage[colorlinks,
linkcolor=blue,
anchorcolor=yellow,
citecolor=red,
]{hyperref}
\allowdisplaybreaks
\numberwithin{equation}{section}
\numberwithin{figure}{section}
\numberwithin{table}{section}

\makeatletter

\providecommand{\tabularnewline}{\\}

\usepackage{paralist}
\theoremstyle{plain}
\newtheorem{thm}{\protect\theoremname}
\theoremstyle{remark}
\newtheorem{rem}[thm]{\protect\remarkname}
\ifx\proof\undefined
\newenvironment{proof}[1][\protect\proofname]{\par
	\normalfont\topsep6\p@\@plus6\p@\relax
	\trivlist
	\itemindent\parindent
	\item[\hskip\labelsep\scshape #1]\ignorespaces
}{%
	\endtrivlist\@endpefalse
}
\providecommand{\proofname}{Proof}
\fi

\journal{Example: Nuclear Physics B}

\@ifundefined{showcaptionsetup}{}{%
 \PassOptionsToPackage{caption=false}{subfig}}
\usepackage{subfig}
\makeatother

\usepackage{caption}
\usepackage{babel}
\providecommand{\remarkname}{Remark}
\providecommand{\theoremname}{Theorem}

\begin{document}

\begin{frontmatter}{}

\title{Sharp-interface limits of the diffuse interface model for two-phase
inductionless magnetohydrodynamic fluids\tnoteref{t1}}

\tnotetext[t1]{Xiaodi Zhang is supported in part by the National Natural Science
Foundation of China under Grant 12201575 and China Postdoctoral Science Foundation
under Grant 2022M722878. }

\author[zzu,zzu1]{Xiaodi Zhang}

\ead{zhangxiaodi@lsec.cc.ac.cn}

\address[zzu]{Henan Academy of Big Data, Zhengzhou University, Zhengzhou 450052,
China.}

\address[zzu1]{School of Mathematics and Statistics, Zhengzhou University, Zhengzhou
450001, China.}

\begin{abstract}
In this paper, we propose and analyze a diffuse interface model for
inductionless magnetohydrodynamic fluids. The model couples a convective
Cahn-Hilliard equation for the evolution of the interface, the Navier\textendash Stokes
system for fluid flow and the possion equation for electrostatics.
The model is derived from Onsager's variational principle and conservation
laws systematically. We perform formally matched asymptotic expansions
and develop several sharp interface models in the limit when the interfacial
thickness tends to zero. It is shown that the sharp interface limit
of the models are the standard incompressible inductionless magnetohydrodynamic
equations coupled with several different interface conditions for
different choice of the mobilities. Numerical results verify the convergence
of the diffuse interface model with different mobilitiess.
\end{abstract}

\end{frontmatter}{}

\section{Introduction\label{sec:Intro}}

Incompressible Magnetohydrodynamics (MHD) is the study of the interaction
between the conducting incompressible fluid flows and electromagnetic
fields. The corresponding model is a system of partial differential
equations, which couples the Navier\textendash Stokes equations and
the Maxwell equations via Lorentz\textquoteright s force and Ohm\textquoteright s
Law. However, in most industrial and laboratory flows, MHD flows typically
occur at small magnetic Reynolds number. In these situations, the
induced magnetic field is usually neglected compared with the applied
magnetic field. The MHD equations with this simplification are known
as the inductionless MHD model. We refer to \citep{Abdou2001,Abdou2005,Davidson2001,Ni2012,Li2019}
for the extensive theoretical modeling, numerical methods and numerical
analysis for this model.

In this paper, we focus on the dynamic behavior of two incompressible,
immiscible and electrically conducting fluids under the influence
of a magnetic field. It has a number of technological and industrial
applications such as fusion reactor blankets, metallurgical industry,
liquid metal magnetic pumps and aluminum electrolysis \citep{Davidson2001,Gerbeau2006,Szek1979,Morl2000}.
In the process of metallurgy and metal material processing, the two-phase
interface problems are involved in the dynamic evolution of air bubbles
and in metal liquid, the shaking of the surface of metal liquid in
the container and the dynamic behavior of floating or attached metal
droplets. In a fusion reactor, two-phase MHD flow is used to describe
the laying process in the liquid-metal cooling blanket \citep{Abdou2001}.
Thus, the interfacial dynamic of two-phase MHD problem has been a
topic of great interest.

The theoretical analysis and numerical simulation of multi-phase flow
is a challenging problem in computational fluid dynamics. A major
effort has been made for studying interfacial dynamic problems in
the past decades. There are two prevalent approaches to the study
of two-phase flows in the literature, sharp interface method and diffuse
interface method. The former is built on the assumption that the fluids
under investigation are completely immiscible, see Fig. \ref{fig:interface}
for an illustration. Thus, a sharp interface (free curve) that will
deform with time exists between the two fluids. This kind of approach
usually leads to a model consists of the bulk equations for each phase
and a set of interfacial balance conditions. The sharp interface models
have been very successful in engineering and scientific applications
\citep{Joseph1993a,Joseph1993b,Sussman2007,Nguyen2018}. In the literature,
several numerical methods have been developed to numerically solve
the sharp-interface model for the two-phase inductionless MHD flows.
Examples of such methods are the front-tracking method \citep{Samulyak2007},
the volume-of-fluid method \citep{Zhang2014c,Zhang2018b}, the level-set
method \citep{Ki2010,Xie2007}. However, there are several known drawbacks
associated with this sharp interface approach. In particular, this
approach is not able or not be convenient to handle the topological
changes such as self-intersection, pinch-off, splitting, and fattening,
and moving contact lines.

\begin{figure}
\begin{centering}
\begin{tabular}{cc}
\includegraphics[scale=0.5]{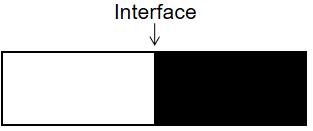} & \includegraphics[scale=0.5]{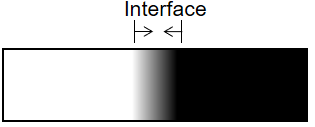}\tabularnewline
\end{tabular}
\par\end{centering}
\caption{Sharp interface(Left) and diffusive interface(Right).\label{fig:interface}}

\end{figure}

An alternative approach, the so-called diffuse interface or phase-field
method is based on the assumption that the macroscopically immiscible
fluids are mixed partially and store the mixing energy within a thin
transition layer \citep{Gurtin1996,Amrouche1998}. Hence, this method
recognizes the microscale mixing and smear the sharp interface into
a transition layer of small thickness $\epsilon$, see Fig. \ref{fig:interface}
for an illustration. To indicate phases, an order parameter or phase
field $\varphi$ is introduced, which varies continuously in the interfacial
region and is mostly uniform in the bulk phases. Utilizing the phase
variable $\varphi$, the after-sought interface can be identified
with the zero-level set of the phase function implicitly. In the diffuse
interface model, the phase-field function usually described by a Cahn-Hilliard
equation \citep{Cahn2013} or an Allen-Cahn equation \citep{Allen1979}.
In contrast to sharp interface method, the phase field method has
features that there is no need to explicitly track the moving interface
and it has advantages in capturing topological changes of the interface.
Thus, diffuse interface method has attracts the most attention among
powerful modeling and numerical tools. About the literature on diffuse
interface method, we also refer to \citep{Liu2009,Liu2013,Liu2015,Abels2008,Abels2012}
for two-phase flows, to \citep{Nochetto2014,Nochetto2016} for ferrohydrodynamics
flows, to \citep{Eck2009,Fontelos2011,Nochetto2014} for electrowetting
problems, to \citep{Lee2002,Han2014,Fei2017} for porous medium problems,
to \citep{Hilhorst2015,Ebenbeck2019,Garcke2016} for tumour growth,
and to \citep{Shen2012,Du2020,Tang2020} for overviews.

In comparison, there are much less works on diffuse interface method
for two-phase incompressible inductionless MHD equations in the literature.
In 2014, Ding and his collaborators \citep{Ding2014} presented a
two-phase inductionless MHD model using phase field method and simulated
the deformation of melt interface in an aluminum electrolytic cell.
In 2020, Chen et al. \citep{Chen2020} proposed two linear, decoupled,
unconditionally energy stable and second order time-marching schemes
to simulate the two-phase conducting flow. In \citep{Mao2021}, Mao
et al. analyzed the fully discrete finite element approximation of
a three-dimensional diffuse interface model for two-phase inductionless
MHD fluids and proved the well-posedness of weak solution to the phase
field model by using the classical compactness method. However, the
diffusion interface model presented are obtained by coupling Cahn-Hilliard
equation and single-phase inductionless MHD equations. This is unfriendly
to incorporate additional different physical effects. This is one
of the motivation of our study.

In phase field model, the interface between the two phases is modeled
via a diffuse interface, which has a thin thickness $\epsilon$. Ideally,
we expect that the thickness of the diffuse interface should be chosen
as small as the physical size. On the one hand, in many applications,
$\epsilon$ is typically nano-scale so that the model is too expensive
to solve it. On the other hand, many numerical experiments have shown
that $\epsilon=10^{-2}\sim10^{-3}$ has already describes the qualitative
features of the dynamic behavior for the flow. Therefore, in numerical
simulations, people often prefer to choose a much larger (than physical
values) interface thickness parameter $\epsilon$. Only when diffuse
interface model approximates a sharp-interface limit accurately, the
numerical simulations with relatively large interface thickness can
be reliable. As a result, after the phase field model is obtained,
one of important and natural issue is to investigate whether the diffuse
interface model can be related to the corresponding sharp interface
model when the interfacial width tends to zero. The sharp interface
limits of some different diffusion interface models can be found in
\citep{Abels2012,Fei2017,Hilhorst2015,Ebenbeck2019,Garcke2016,Feng2004a,Ebenbeck2020,Aki2014}
for instance. Nevertheless, up to our knowledge, there seems to exist
no corresponding rigorous results on sharp interface limits of diffuse
interface method for two-phase incompressible inductionless MHD equations.
This is the main motivation of our study.

The first objective of this paper is to derive a diffuse interface
model model for two-phase incompressible inductionless MHD flows with
the matched density systematically. Using Onsager\textquoteright s
variational principle combined with conservation laws, we deduce a
thermodynamically consistent model which couples Cahn-Hilliard equation
for the phase field and chemical potential, Navier-Stokes equations
for the velocity and pressure, and Poisson equations for the current
density and electric potential. In the model, these equations are
nonlinearly coupled through convection, stresses, generalized Ohm\textquoteright s
law and Lorentz forces. The appealing variational-based formalism
of the model, make it facilitate the inclusion of different physical
effects.

The second objective of the paper is to investigate the sharp-interface
limits of the diffuse interface models for several setups of three
types mobilities. This is done by using the method of formally matched
asymptotic expansions and numerical simulations. In asymptotic analysis,
we consider several typical configurations of mobilities $M\left(\varphi\right)$
in the model. When $M\left(\varphi\right)=\epsilon m_{0}$ tends to
zero or $M\left(\varphi\right)=m_{0}\left(1-\varphi\right)_{+}^{2}$
degenerates in the bulk, we show that the sharp interface problems
are standard free boundary problems for inductionless MHD system.
In the bulk, we obtain the inductionless MHD equations. On the interface,
the velocity is continuous and the stress fulfills the Yong-Laplace
law, and the current density is normal-continuous and the electric
potential is continuous. Moreover, the interface is only transported
by the velocity of the fluid. While in the case of a constant mobility
$M\left(\varphi\right)=m_{0}$, we show that the sharp interface problems
in the bulk are inductionless MHD equations and a harmonic equation
for the chemical potential. The interface is no longer material and
its evolution is related to both the normal velocity of the fluid
and the jump of the flux of the chemical potential on the interface(Stefen
type condition). On the interface, the chemical potential is a constant
related the surface tension coefficient and the mean curvature of
the interface. The remaining interface conditions are the same as
in the former two case. Furthermore, numerical experiments verify
the convergence of the diffuse interface model as the thickness of
the interfacial layer $\epsilon$ tends to zero.

The rest of the paper is organized as follows. In Section 2, we derive
the diffuse interface model based on Onsager\textquoteright s variational
principle and give the formal energy estimates. In Section 3, we utilize
the method of formally matched asymptotic to derive sharp interface
limits for different mobilities and prove the formal energy estimates
for the sharp interface models. In Section 4, we present some numerical
examples which verify the convergence of the diffuse-interface model.
In Section 5, some conclusion and remarks are presented.

\section{Model Derivation\label{sec:Model}}

In this section we derive a mathematical model for the flow of diffuse
interface incompressible inductionless magnetohydrodynamic fluids.
The approach is based on Onsager's variational principle \citep{Lars1931,Onsager1931,Onsager1953},
which usually yields thermodynamically consistent diffuse interface
models. The procedure is quite similar to \citep{Abels2012,Eck2009}. 

We consider phase-field models for a mixture of two immiscible, incompressible
and conducting fluids with the matched density in a bounded domain
$\Omega$ with Lipschitz-continuous boundary $\Sigma\coloneqq\partial\Omega$
in $\mathbb{R}^{{\rm d}},{\rm d}=2,3$. In order to present the diffuse
interface model, one assumes a partial mixing of the macroscopically
immiscible fluids in a thin interfacial region. We begin by introducing
a phase function (macroscopic fluid labeling function) $\varphi$ such
that
\[
\varphi(x,t)=\left\{ \begin{array}{ll}
-1 & \text{ fluid }1\\
1 & \text{ fluid }2
\end{array}\right.
\]
 with a thin, smooth transition region of width $O(\epsilon)$. Surface
energy is included into phase-field models by a contribution to the
free energy of the following Helmholtz free energy functional, the
tendencies for mixing and de-mixing are in competition 
\begin{equation}
E_{\varphi}=\int_{\Omega}\gamma\left(\frac{\epsilon}{2}|\nabla\varphi|^{2}+\frac{1}{\epsilon}F(\varphi)\right)\label{eq:surfaceen}
\end{equation}
where $F(\varphi)$ is the double-well potential (chemical
energy density) with minima at $\ensuremath{\pm}1$, $\gamma$ is
the surface tension, and $\epsilon$ is the interface thickness. The
first term, the gradient energy contributes to the hydrophilic type
of interactions, the bulk energy represents the hydrophobic type of
interactions. There are two popular choices of the double-well potential
$F(\varphi)$, 
\begin{itemize}
\item Ginzburg\textendash Landau double-well potential \citep{Novick2008},
see Fig. \ref{fig:poten},
\[
F(\varphi)=\frac{1}{4}\left(\varphi^{2}-1\right)^{2}.
\]
\item Flory\textendash Huggins logarithmic potential \citep{Miranville2019},
see Fig. \ref{fig:poten},
\[
F(\varphi)=\frac{1+\varphi}{2}\ln\left(\frac{1+\varphi}{2}\right)+\frac{1-\varphi}{2}\ln(\frac{1-\varphi}{2})+\frac{\theta}{4}\left(\varphi^{2}-1\right)^{2},
\]
where $\theta>2$ is the energy parameter.
\end{itemize}

\begin{figure}
\begin{centering}
\begin{tabular}{cc}
\includegraphics[scale=0.5]{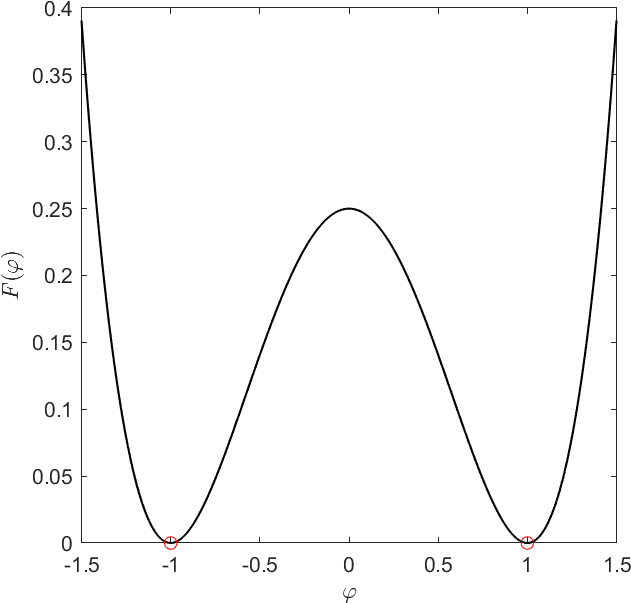} & \includegraphics[scale=0.5]{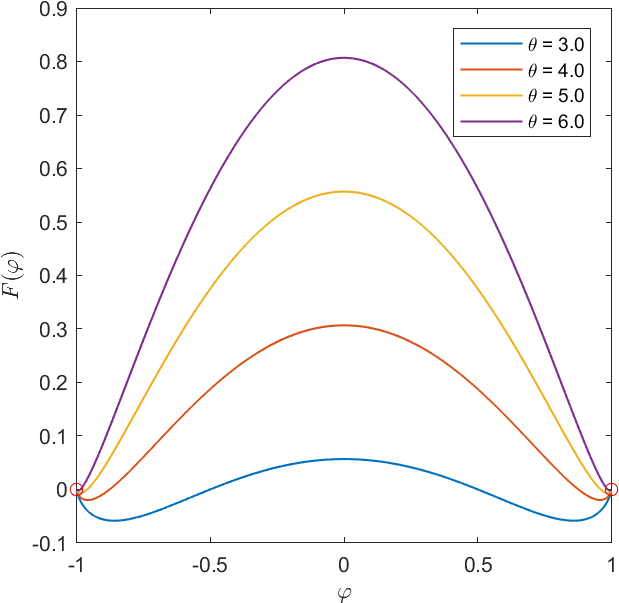}\tabularnewline
\end{tabular}
\par\end{centering}
\caption{Ginzburg\textendash Landau double-well potential(Left) and Flory\textendash Huggins
logarithmic potential(Right).\label{fig:poten}}
\end{figure}

\noindent Indeed, the double-well potential penalizes sharp transition and helps
to create the transition layer depicted in Fig. \ref{fig:interface}.
The equilibrium configuration is the consequence of the competition
between the two types of interactions. Let $\boldsymbol{u}$ be the
velocity of incompressible fluids, the phase field $\varphi$ will
evolve in time following a convective Cahn-Hilliard equation 
\begin{equation}
\partial_{t}\varphi+{\rm div}(\varphi\boldsymbol{u})-{\rm div}\boldsymbol{J}_{\varphi}=0\label{model:chd}
\end{equation}
where $\boldsymbol{J}_{\varphi}$ denotes the mass flux will be found
later. Here we have used the incompressibility of fluids ${\rm div}\boldsymbol{u}=0$
to rewrite advection operator $(\boldsymbol{u}\cdot\nabla)\varphi$
in nonconservative form as the conservative form ${\rm div}(\varphi\boldsymbol{u})$.

The electromagnetic field adds to the total free energy a contribution
essentially \citep{Castellanos1998,Hanson2002},
\[
E_{\boldsymbol{B},\boldsymbol{D}}=\frac{1}{2}\int_{\Omega}\frac{\boldsymbol{D}^{2}}{\varepsilon_{r}(\varphi)}+\frac{1}{2}\int_{\Omega}\frac{\boldsymbol{B}^{2}}{\mu_{r}(\varphi)},
\]
where $\varepsilon_{r}(\varphi)$ is the dielectric constant,and
$\boldsymbol{D}$ is the (electric) displacement vector, $\mu_{r}(\varphi)$
is the magnetic permeability and $\boldsymbol{B}$ is the magnetic
displacement vector. However, since the displacement current and the
inducted magnetic field are often neglected in the inductionless MHD
\citep{Davidson2001,Gerbeau2006,Ni2007}, there is no contribution
to the total energy from both the electric and magnetic field. With
these simplifications, Maxwell equations are replaced by Poisson equation
for the electric scalar potential,
\begin{align}
\sigma(\varphi)^{-1}\boldsymbol{J}+\nabla\phi-\boldsymbol{u}\times\boldsymbol{B} & =\boldsymbol{0},\label{model:po}\\
{\rm div}\boldsymbol{J} & =0,\nonumber 
\end{align}
where $\boldsymbol{B}$ is the applied magnetic field which
is assumed to be given. It is worth remarking that the first equation of (\ref{model:po}) is obtained by combing generalized Ohm\textquoteright s law with $\boldsymbol{E}=-\nabla \phi$, see \cite{Li2019,Ni2007}. 

The inertia of the fluid would add a kinetic energy to the free energy
of the form,
\[
E_{\boldsymbol{u}}=\frac{1}{2}\int_{\Omega}\rho|\boldsymbol{u}|^{2},
\]
 where $\rho$ is the density of the fluid material. We model the
incompressible flow by the following general evolution equations,
\begin{align}
\rho\left(\partial_{t}\boldsymbol{u}+\boldsymbol{u}\cdot\nabla\boldsymbol{u}\right)-{\rm div}\mathbf{S}+\nabla p& =\boldsymbol{F},\label{model:ns}\\
{\rm div}\boldsymbol{u} & =0.\nonumber 
\end{align}
 where $p$ is the pressure, $\boldsymbol{F}$ is the field of forces exerted on the fluids, $\mathbf{S}$ denotes the symmetric
stress tensor. In addition, we denote  $\boldsymbol{\mathrm{T}}\coloneqq-p\boldsymbol{\mathrm{I}}+\mathbf{S}$ by stress. The expressions of $\boldsymbol{F}$ and $\mathbf{S}$
will be specified later. In this work, we consider two-phase inductionless
magnetohydrodynamic fluids with matched density $\rho_{1}=\rho_{2}$. The extension to the non-matched density case will be discussed in Remark \ref{rem:density}. Without lose of generality, we will set $\rho\equiv1$ in the rest
of this paper. 

Thus, the total free energy is given as a sum of the interfacial energy
$E_{\varphi}$ and the kinetic energy $E_{\boldsymbol{u}}$,
\[
E=E_{\varphi}+E_{\boldsymbol{u}}.
\]

We assume the velocity $\boldsymbol{u}$ and the fluxes $\boldsymbol{J}_{\phi}\cdot\boldsymbol{n}$,
$\boldsymbol{J}\cdot\boldsymbol{n}$ to vanish at $\Sigma$, where
$\boldsymbol{n}$ is the outer unit normal of $\Sigma$. The chemical
potential $\mu$ is defined by the variational derivative of the energy
$E$ with respect $\varphi$,
\[
\mu\coloneqq\frac{\delta E}{\delta\varphi}=-\gamma\varepsilon\Delta\varphi+\frac{\gamma}{\varepsilon}f(\varphi).
\]

With this notation, the time derivative of the free energy
is given by 
\[
\frac{\mathrm{d}E}{\mathrm{d}t}=\int_{\Omega}\boldsymbol{u}\cdot\frac{\partial\boldsymbol{u}}{\partial t}+\int_{\Omega}\mu\frac{\partial\varphi}{\partial t}.
\]
To deal with the divergence-free constraints for velocity and current density,
we regard the pressure and electric potential as the corresponding Lagrange multipliers. Then we have
\begin{equation}
\frac{\mathrm{d}E}{\mathrm{d}t}=\int_{\Omega}\boldsymbol{u}\cdot\partial_{t}\boldsymbol{u}+\int_{\Omega}\mu\partial_{t}\varphi-\int_{\Omega}p\nabla\cdot\boldsymbol{u}-\int_{\Omega}\phi\nabla\cdot\boldsymbol{J}.\label{eq:Energyid}
\end{equation}
Next we rewrite some terms on the right hand side of (\ref{eq:Energyid}).
For the first and third terms, inserting the equation (\ref{model:ns}), integrating
by parts and using the fact that the normal component of the velocity
vanish at $\Sigma$ gives,
\begin{equation}
\int_{\Omega}\boldsymbol{u}\cdot\partial_{t}\boldsymbol{u}=-\int_{\Omega}\mathbf{S}:D\left(\boldsymbol{u}\right)+\int_{\Omega}\boldsymbol{u}\cdot\boldsymbol{F}.\label{eq:enns}
\end{equation}

\noindent where $D\left(\boldsymbol{u}\right)\coloneqq\frac{\nabla\boldsymbol{u}+\nabla\boldsymbol{u}^{{\rm T}}}{2}$
is strain velocity tensor. Analogously, inserting the equation (\ref{model:chd}),
integrating by parts and using the fact that $\boldsymbol{J}_{\varphi}\cdot\boldsymbol{n}$
vanish at $\Sigma$ gives,
\begin{equation}
\int_{\Omega}\mu\partial_{t}\varphi=-\int_{\Omega}\boldsymbol{J}_{\varphi}\cdot\nabla\mu+\int_{\Omega}\varphi\nabla\mu\cdot\boldsymbol{u}.\label{eq:eqch}
\end{equation}

\noindent For the last term, integrating
by parts, inserting the equation (\ref{model:po}) and using the fact that $\boldsymbol{J}\cdot\boldsymbol{n}$
vanish at $\Sigma$, we have
\begin{equation}
\int_{\Omega}\phi\nabla\cdot\boldsymbol{J}=-\int_{\Omega}\nabla\phi\cdot\boldsymbol{J}=
\int_{\Omega}\sigma(\varphi)^{-1}\boldsymbol{J}\cdot\boldsymbol{J}-\int_{\Omega}\boldsymbol{u}\times\boldsymbol{B}\cdot\boldsymbol{J}.\label{eq:eqpo}
\end{equation}

\noindent Substituting (\ref{eq:enns})-\eqref{eq:eqpo} into (\ref{eq:Energyid}),
the time-derivative of the free energy can be rewritten as
\begin{equation}
\begin{aligned}\frac{\mathrm{d}E}{\mathrm{d}t}= & -\int_{\Omega}\mathbf{S}:D(\boldsymbol{u})-\int_{\Omega}\boldsymbol{J}_{\varphi}\cdot\nabla\mu-\int_{\Omega}\sigma(\varphi)^{-1}\boldsymbol{J}\cdot\boldsymbol{J}\\
 & +\int_{\Omega}\boldsymbol{F}\cdot\boldsymbol{u}+\int_{\Omega}\varphi\nabla\mu\cdot\boldsymbol{u}-\int_{\Omega}\boldsymbol{J}\times\boldsymbol{B}\cdot\boldsymbol{u}.
\end{aligned}
\label{eq:eqer}
\end{equation}

Collecting all terms having a scalar product with the velocity in
(\ref{eq:eqer}) gives the rate of change of the mechanical work with
\[
\frac{\mathrm{d}W}{\mathrm{d}t}=\int_{\Omega}\boldsymbol{F}\cdot\boldsymbol{u}+\int_{\Omega}\varphi\nabla\mu\cdot\boldsymbol{u}-\int_{\Omega}\boldsymbol{J}\times\boldsymbol{B}\cdot\boldsymbol{u}.
\]

\noindent Since no external forces are included into the system, we
conclude that 
\[
\frac{\mathrm{d}W}{\mathrm{d}t}=0.
\]
Therefore, we find an expression for the forces $\boldsymbol{F}$
acting on the fluid, 
\[
\boldsymbol{F}=-\varphi\nabla\mu+\boldsymbol{J}\times\boldsymbol{B}.
\]

\noindent The first term $-\varphi\nabla\mu$ corresponds to surface
tension, and the second term is the Lorentz force. 

%

To determine the fluxes $\boldsymbol{J}_{\varphi}$ and the stress tensor $\mathbf{S}$, we turn to Onsager's variational
principle. We introduce the dissipation functional, which is the
sum of quadratic in the fluxes with phenomenological parameters,
\begin{equation}
\Phi\left(\mathbf{J},\mathbf{J}\right)=\int_{\Omega}\frac{\left|\boldsymbol{J}_{\varphi}\right|^{2}}{2M(\varphi)}+\int_{\Omega}\frac{|\boldsymbol{J}|^{2}}{2\sigma(\varphi)}+\int_{\Omega}\frac{|\mathbf{S}|^{2}}{2\eta(\varphi)}\label{eq:disp}
\end{equation}
 where the fluxes are $\mathbf{J}=\left(\mathbf{S},\boldsymbol{J}_{\varphi},\boldsymbol{J}\right)$.
In (\ref{eq:disp}), $M(\varphi)$ is the diffusional mobility related
to the relaxation time scale, $\eta(\varphi)$ is the viscosity of
the fluids and $\sigma(\varphi)$ is the electric conductivity. It
means that the dissipation of $E$ stems from the friction losses,
the Ohmic losses and the diffusion transport. The Onsager\textquoteright s
relation yields,

\noindent 
\begin{equation}
\delta_{\mathbf{J}}\left(\frac{\mathrm{d}E}{\mathrm{d}t}+\Phi\left(\mathbf{J},\mathbf{J}\right)\right)=0\label{eq:varf}
\end{equation}
Solving the variational problem (\ref{eq:varf}), we obtain the fluxes 
\[
\boldsymbol{J}_{\varphi}=M(\varphi)\nabla\mu,\quad\mathbf{S}=2\eta(\varphi)D\left(\boldsymbol{u}\right).
\]
 These ``constitutive'' relations can be also obtained by assuming
that the fluxes depend linearly on the thermodynamic forces. In fact,
it is implicitly postulated in the form of the dissipation function
$\Phi$.

To summarize, we obtain the incompressible Cahn\textendash Hilliard\textendash Inductionless-MHD
system,\label{model:chimhd}
\begin{align}
\boldsymbol{u}_{t}+\boldsymbol{u}\cdot\nabla\boldsymbol{u}-2\nabla\left(\eta\left(\varphi\right)D\left(\boldsymbol{u}\right)\right)+\nabla p-\boldsymbol{J}\times\boldsymbol{B}+\varphi\nabla\mu & =\boldsymbol{0},\label{model:fluid}\\
\mathrm{div}\boldsymbol{u} & =0,\label{model:divu}\\
\sigma\left(\varphi\right)^{-1}\boldsymbol{J}+\nabla\phi-\boldsymbol{u}\times\boldsymbol{B} & =\boldsymbol{0},\label{model:poisson}\\
\mathrm{div}\boldsymbol{J} & =0,\label{model:divJ}\\
\varphi_{t}+{\rm div}(\varphi\boldsymbol{u})-{\rm div}\left(M\left(\varphi\right)\nabla\mu\right) & =0,\label{model:ch}\\
-\gamma\varepsilon\Delta\varphi+\frac{\gamma}{\varepsilon}f(\varphi)-\mu & =0,\label{model:chc}
\end{align}
\noindent with the following initial and boundary conditions
\begin{align}
\boldsymbol{u}(0)=\boldsymbol{u}^{0},\quad\varphi(0) & =\varphi^{0},\label{model:init}\\
\boldsymbol{u}=0,\quad\boldsymbol{J}\cdot\boldsymbol{n} & =0,\label{model:uJb}\\
\partial_{\boldsymbol{n}}\varphi=0,\quad M\left(\varphi\right)\partial_{\boldsymbol{n}}\mu & =0.\label{model:chb}
\end{align}

The dynamic behavior of the flow is governed by the coupling model
of the Cahn-Hilliard equations, the Navier-Stokes equations and the
Poisson equation. As mentioned earlier, the last term on the left-hand
side of (\ref{model:fluid}) the continuum surface tension force in
the potential form. This term appears differently in different literature
\citep{Liu2013,Liu2015,Shen2012}: $-\nabla\cdot(\nabla\varphi\otimes\nabla\varphi),$
$-\mu\nabla\varphi,$ or $\varphi\nabla\mu.$ It can be shown that
the three expressions are equivalent by redefining the pressure $p$.
To deal with the discontinuous of material properties, we assume that
all the material properties $\Psi\in\left\{ \eta\left(\varphi\right),\sigma\left(\varphi\right)^{-1}\right\} $
that depend on the phase is a Lipschitz-continuous function of $\varphi$
satisfying 
\[
0<\min\left\{ \Psi_{1},\Psi_{2}\right\} \le\Psi(\phi)\le\max\left\{ \Psi_{1},\Psi_{2}\right\} .
\]
where $\Psi_{i}$ is the material parameter of the fluid $i$.
\begin{rem}
\label{rem:density}
Though the model considered is for match-density two-phase system,
one can still employ this with Boussinesq approximation in the case
of small density ratio \citep{Liu2013}. Specifically, a gravitational
force are supplemented to the right hand side of the momentum equations
(\ref{model:fluid}) to model the effect of density difference. The
modeling and analysis for large density ratio case are ongoing.
\end{rem}
\begin{rem}
The model that we derived can be generalized to handle situations
that requires other physical factors easily. As a representative example,
we present the sketch of an extension for the model (\ref{model:chimhd})
with moving contact lines in \ref{sec:CHiMHDMCL}. To more general
situation when either the system is subjected to gravitational forces
or when one additional soluble species is present in both fluids,
we refer \textcolor{black}{to }\citep{Abels2012,Engblom2013}.
\end{rem}
To end this section, we present the energy law for the CHiMHD system.
\begin{thm}
\label{thm:moengy}Let $(\boldsymbol{u},p,\boldsymbol{J},\phi,\varphi,\mu)$
be the solution of (\ref{model:chd}). The energy law holds
\begin{equation}
\frac{\mathrm{d}E}{\mathrm{d}t}=-\Phi.\label{eq:moengylaw}
\end{equation}
where 
\begin{align*}
E & =\int_{\Omega}\left(\frac{1}{2}|\boldsymbol{u}|^{2}+\frac{\lambda\epsilon}{2}|\nabla\varphi|^{2}+\frac{\lambda}{\epsilon}F(\varphi)\right),\\
\Phi & =\int_{\Omega}\left(2\eta(\varphi)\left|D(\boldsymbol{u})\right|^{2}+M(\varphi)\left|\nabla\mu\right|^{2}+\sigma(\varphi)^{-1}|\boldsymbol{J}|^{2}\right).
\end{align*}
\end{thm}
\begin{proof}
An immediate consequence of (\ref{eq:eqer}) gives the theorem.
\end{proof}
The energy law describes the variation of the total energy caused
by energy conversion, which is crucial to the analysis of and the
design of unconditionally stable numerical schemes for the model.

\section{Sharp Interface Asymptotic\label{sec:Sharp}}

In this section, we perform formally matched asymptotic expansions
to derive the sharp-interface limits of a phase-field model. The idea
of the method is to plug the outer and inner expansions into the model
equations and solve them order by order, in addition these expansions
should match up \citep{Lagerstrom1988,Caginalp1988,Fife1995}. To
make the presentation clear, we adopt $\omega_{\varepsilon}\in\left\{ \boldsymbol{u}_{\varepsilon},p_{\varepsilon},\boldsymbol{J}_{\epsilon},\phi_{\epsilon},\varphi_{\varepsilon},\mu_{\varepsilon}\right\} $
instead of $\omega\in\left\{ \boldsymbol{u},p,\boldsymbol{J},\phi,\varphi,\mu\right\} $
in the system (\ref{model:chimhd}) to stress that these functions
depend on $\varepsilon$. We use the notation $(\cdot)_{o}^{a}$ and
$(\cdot)_{I}^{a}$ for the terms resulting from the order $a$ outer
and inner expansions of (\ref{model:chimhd}), respectively. 

For any $\epsilon>0$ and for each time $t\in[0,T],$given a function
$\varphi_{\epsilon}\left(x,t\right)$, we define its zero-level sets
as 
\[
\Gamma_{\epsilon}:=\left\{ (x,t)\in\Omega\times[0,T]:\varphi_{\epsilon}(x,t)=0\right\} .
\]

\noindent Two open subdomains are separated by $\Gamma_{\epsilon}$
are denoted by 

\[
\Omega_{+}(\epsilon,t):=\left\{ x\in\Omega:\varphi_{\epsilon}(x,t)>0\right\} ,\quad\Omega_{-}(\epsilon,t):=\left\{ x\in\Omega:\varphi_{\epsilon}(x,t)<0\right\} .
\]

We consider the model (\ref{model:chimhd}) with the following choices
and assumptions \citep{Garcke2016,Ebenbeck2020,Garcke2006}.
\begin{enumerate}
\item \label{assum:inter}We assume that $\Gamma_{\epsilon}\subseteq\Omega$
are closed evolving hyper-surface, which depend smoothly on $t$ and
$\epsilon$ and converge as $\epsilon\rightarrow0$ to a limiting
evolving hyper-surface $\Gamma$ which moving with normal velocity
$\mathcal{V}$. Further, for every $\epsilon>0$, we assume that the
domain $\Omega$ can be divided into two subdomains $\Omega_{+}(\epsilon,t)$
and $\Omega_{-}(\epsilon,t)$, and $\Omega_{+}(\epsilon,t)$ is the
region enclosed by $\Gamma_{\epsilon}$.
\item \label{assum:expan}The solutions to (\ref{model:chimhd}) $\left(\boldsymbol{u}_{\varepsilon},p_{\varepsilon},\boldsymbol{J}_{\epsilon},\phi_{\epsilon},\varphi_{\varepsilon},\mu_{\varepsilon}\right)_{\epsilon>0}$
are assumed to be sufficiently smooth so that they have an asymptotic
expansion with respect to $\epsilon$ in the bulk regions away from
$\Gamma$ (outer expansion), and another expansion in the interfacial
region close to $\Gamma$ (inner expansion).
\item \label{assum:mob}For the mobility $m_{\varepsilon}$, we distinguish
three cases:
\begin{equation}
M_{\varepsilon}(\varphi)=\left\{ \begin{array}{ll}
m_{0} & \text{ case I },\\
\varepsilon m_{0} & \text{ case II },\\
m_{0}\left(1-\varphi^{2}\right)_{+} & \text{ case III },
\end{array}\right.\label{eq:mob}
\end{equation}
where $m_{0}$ is a positive constant, $(\cdot)_{+}$ is defined by
\[
(a)_{+}=\begin{cases}
a & a\ge0,\\
0 & a<0.
\end{cases}
\]
\end{enumerate}
To make the presentation in this section clear, we only present the
sharp interface asymptotic analysis for the Ginzburg-Landau double-well
potential $F(\varphi)=\frac{1}{4}\left(\varphi^{2}-1\right)^{2}$.
The extension to general potentials is trivial. We omit the details
but give a sketch in Remark \ref{rem:GeneralPoten}.

\subsection{Outer expansion\label{subsec:Outer}}

We first expand the solution in outer regions away from the interface.
For $\omega_{\varepsilon}\in\left\{ \boldsymbol{u}_{\varepsilon},p_{\varepsilon},\boldsymbol{J}_{\epsilon},\phi_{\epsilon},\varphi_{\varepsilon},\mu_{\varepsilon}\right\} ,$
the assumption \ref{assum:expan} implies the following outer expansions,
\[
\omega_{\varepsilon}=\omega_{0}+\varepsilon\omega_{1}+\varepsilon^{2}\omega_{2}+\cdots.
\]

\noindent We substitute the above expansions to the CHiMHD system.
The leading order, $\text{\eqref{model:ch}}_{O}^{-1}$ gives 

\begin{equation}
f\left(\varphi_{0}\right)=0\label{out:last}
\end{equation}

\noindent The stable solutions of (\ref{out:last}) corresponding
to the minima of $F$ are $\varphi_{0}=\pm1.$ Thus, we can define
the domain occupied of fluid 1 and fluid 2 by 
\[
\Omega_{-}:=\left\{ x\in\Omega:\varphi_{0}(x)=-1\right\} ,\quad\Omega_{+}:=\left\{ x\in\Omega:\varphi_{0}(x)=1\right\} .
\]

\noindent A sketch of the situation is shown in Fig. \ref{fig:region}.
\begin{center}
\begin{figure}
\centering{}\includegraphics[scale=0.3]{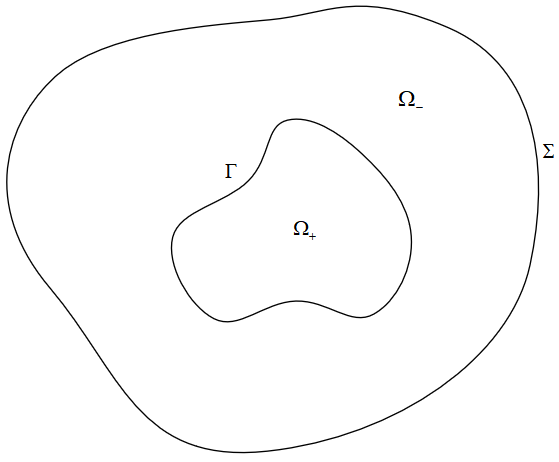}\caption{The domain occupied of fluid 1 and fluid 2.\label{fig:region}}
\end{figure}
\par\end{center}

\noindent For the equations (\ref{model:divu}) and (\ref{model:divJ}),
in the leading order, $\text{\eqref{model:divu}}_{O}^{-1}$ and $\text{\eqref{model:divJ}}_{O}^{-1}$,
we have
\[
\nabla\cdot\boldsymbol{u}_{0}=0,\quad\nabla\cdot\boldsymbol{J}_{0}=0.
\]

Since $\nabla\varphi_{0}=\boldsymbol{0}$ in $\Omega_{+}\cup\Omega_{-}$
, we obtain for the equations in $\Omega_{+}\cup\Omega_{-}$ to zeroth
order that
\begin{equation}
\begin{aligned}\partial_{t}\boldsymbol{u}_{0}+\left(\boldsymbol{u}_{0}\cdot\nabla\right)\boldsymbol{u}_{0}-2\eta_{i}{\rm div}\left(D\left(\boldsymbol{u}_{0}\right)\right)+\nabla p_{0}-\boldsymbol{J}_{0}\times\boldsymbol{B} & =\boldsymbol{0},\\
{\rm div}\boldsymbol{u}_{0} & =0,\\
\sigma_{i}^{-1}\boldsymbol{J}_{0}+\nabla\varphi_{0}-\boldsymbol{u}_{0}\times\boldsymbol{B} & =\boldsymbol{0},\\
\mathrm{div}\boldsymbol{J}_{0} & =0.
\end{aligned}
\label{out:model}
\end{equation}

\noindent Note that this is the standard inductionless MHD equations
in each subdomain. 

We now analyze (\ref{model:chc}) according to the mobilities introduced
in (\ref{eq:mob}). For the cases of $M(\varphi)=\epsilon m_{0}$
and $M(\varphi)=m_{0}\left(1-\varphi^{2}\right)_{+}$, the chemical
potential does not contribute to the equations at zeroth order. In
this case of $M(\varphi)=m_{0}$, we obtain a Poisson equation for
$\mu_{0}$,
\[
\Delta\mu_{0}=0.
\]

\subsection{Inner Expansion}

We now do an expansion in an interfacial region. This subsection uses
ideas presented in \citep{Abels2012,Garcke2016,Xu2018,Wang2007}. 

\subsubsection{New Coordinates and matching conditions}

We first introduce new coordinates in a neighborhood of $\Gamma$.
To this end, we consider the signed distance function $d(x,t)$ of
a point $x$ to $\Gamma$ with $d(x,t)>0$ if $x\in\Omega_{+}$ and
$d(x,t)<0$ if $x\in\Omega_{-}.$ Note that it is well-defined near
the interface and $\mathcal{V}=-\partial_{t}d$. By $\boldsymbol{\nu}=\nabla d$
we denote the unit normal to $\Gamma$ pointing into the region $\Omega_{+}$.
Introduce a new re-scaled variable 

\[
\xi=\frac{d(x,t)}{\varepsilon}.
\]
In a tubular neighborhood of $\Gamma$, for any function $\omega(x,t)$,
we can rewrite it as 
\[
\omega(x,t)=\tilde{\omega}(x,\xi,t).
\]
In this new $(t,x,\xi)$-coordinate system, the following change of
variables apply, 
\[
\begin{array}{l}
\nabla\omega=\nabla\tilde{\omega}+\varepsilon^{-1}\partial_{\xi}\tilde{\omega}\boldsymbol{\nu},\\
\Delta\omega=\Delta\tilde{\omega}-\varepsilon^{-1}\partial_{\xi}\tilde{\omega}\kappa+2\varepsilon^{-1}(\boldsymbol{\nu}\cdot\nabla)\partial_{\xi}\tilde{\omega}+\varepsilon^{-2}\partial_{\xi\xi}\tilde{\omega},\\
\partial_{t}\omega=\partial_{t}\tilde{\omega}+\varepsilon^{-1}\partial_{\xi}\tilde{\omega}\partial_{t}d,
\end{array}
\]
where $\kappa=-\nabla\cdot\boldsymbol{\nu}$ is the mean curvature
of the interface. In particular, if $\boldsymbol{\omega}$ is a vector-valued
function for $x$ in a tubular neighborhood of $\Sigma_{0},$ then
we obtain 
\begin{align*}
{\rm div}\boldsymbol{\omega} & ={\rm div}\tilde{\boldsymbol{\omega}}+\varepsilon^{-1}\partial_{\xi}\tilde{\boldsymbol{\omega}}\cdot\boldsymbol{\nu},\\
\nabla\boldsymbol{\omega} & =\nabla\tilde{\boldsymbol{\omega}}+\varepsilon^{-1}\partial_{\xi}\tilde{\boldsymbol{\omega}}\otimes\boldsymbol{\nu},\\
D\left(\boldsymbol{\omega}\right) & =\mathcal{E}\left(\nabla\tilde{\boldsymbol{\omega}}\right)+\varepsilon^{-1}\mathcal{E}\left(\partial_{\xi}\tilde{\boldsymbol{\omega}}\otimes\boldsymbol{\nu}\right),
\end{align*}

\noindent where $\mathcal{E}\left(\boldsymbol{A}\right)\coloneqq\left(\boldsymbol{A}+\boldsymbol{A}^{{\rm T}}\right)/2$.

We denote the variables $\omega_{\epsilon}\in\left\{ \boldsymbol{u}_{\varepsilon},p_{\varepsilon},\boldsymbol{J}_{\epsilon},\phi_{\epsilon},\varphi_{\varepsilon},\mu_{\varepsilon}\right\} $
in the new coordinate system by $\tilde{\omega}_{\varepsilon}\in\left\{ \tilde{\boldsymbol{u}}_{\varepsilon},\tilde{p}_{\varepsilon},\tilde{\boldsymbol{J}}_{\epsilon},\tilde{\phi}_{\epsilon},\tilde{\varphi}_{\varepsilon},\tilde{\mu}_{\varepsilon}\right\} $.
The assumption means that they have the following inner expansions,
\[
\tilde{\omega}_{\varepsilon}=\tilde{\omega}_{0}+\varepsilon\tilde{\omega}_{1}+\varepsilon^{2}\tilde{\omega}_{2}+\cdots.
\]
 The assumption that the zero level sets of $\varphi_{\varepsilon}$
converge to $\Gamma$ implies that 
\begin{equation}
\tilde{\varphi}_{0}(x,0,t)=0.\label{inner:level}
\end{equation}

In matched asymptotic techniques, inner and outer quantities are linked
together by matching conditions. For this end, we employ the matching
conditions \citep{Garcke2016,Ebenbeck2020,Garcke2006}:
\begin{align}
\underset{\xi\rightarrow\pm\infty}{\lim}\tilde{\omega}_{0}(x,\xi,t) & =\omega_{0}^{\pm}(x),\label{match:a}\\
\underset{\xi\rightarrow\pm\infty}{\lim}\partial_{\xi}\tilde{\omega}_{0}(x,\xi,t) & =0,\label{match:b}\\
\underset{\xi\rightarrow\pm\infty}{\lim}\left(\partial_{\xi}\tilde{\omega}_{1}(x,\xi,t)+\nabla\tilde{\omega}_{0} \right) & =\nabla\omega_{0}^{\pm}(x)\cdot\boldsymbol{\nu}.\label{match:c}
\end{align}
where $\omega^{\pm}$ denotes the restriction of a function $\omega$
in $\Omega_{+}$ and $\Omega_{-}$ respectively. Moreover, we denote
the jump of a quantity $\omega$ across the interface by 
\[
\left[\omega\right]:=\omega^{+}(x,t)-\omega^{-}(x,t).
\]

\subsubsection{The equations to leading order}

We begin with considering the equation (\ref{model:chc}). The leading
order $\eqref{model:chc}_{I}^{-1}$ yields 
\begin{equation}
\partial_{\xi\xi}\tilde{\varphi}_{0}-f\left(\tilde{\varphi}_{0}\right)=0\label{inner:ch}
\end{equation}
 Invoking with (\ref{inner:level}) we obtain that $\tilde{\varphi}_{0}$
is independent of $x$ and $t$. Thus, we write $\tilde{\varphi}_{0}^{\prime}(\xi)$
and $\tilde{\varphi}_{0}^{\prime\prime}(\xi)$ instead of $\partial_{\xi}\tilde{\varphi}_{0}$
and $\partial_{\xi\xi}\tilde{\varphi}_{0}$. As a consequence, $\tilde{\varphi}_{0}$
is only a function of $\xi$, and fulfills 
\begin{equation}
\tilde{\varphi}_{0}^{\prime\prime}(\xi)-f\left(\tilde{\varphi}_{0}(\xi)\right)=0,\quad\tilde{\varphi}_{0}(0)=0,\quad\underset{\xi\rightarrow\pm\infty}{\lim}\tilde{\varphi}_{0}(\xi)=\pm1,\label{inner:ch-1}
\end{equation}
where we have employed the match condition (\ref{match:a}). For the
Ginzburg\textendash Landau double-well potential, the unique solution
is given by
\[
\tilde{\varphi}_{0}(\xi)=\tanh\left(\xi/\sqrt{2}\right).
\]
Multiplying (\ref{inner:ch-1}) by $\tilde{\varphi}_{0}^{\prime}(\xi)$,
integrating and applying the matching conditions (\ref{match:a})
and (\ref{match:b}) to $\tilde{\varphi}_{0}(\xi)$, we obtain the
equipartition of energy \citep{Bowler1982}
\[
\frac{1}{2}\left|\tilde{\varphi}_{0}^{\prime}(\xi)\right|^{2}=F\left(\tilde{\varphi}_{0}(\xi)\right)\text{ for all }|\xi|<\infty.
\]

\noindent Consequently, we have 
\[
\begin{aligned}\iota\coloneqq\int_{-\infty}^{\infty}\left|\partial_{\xi}\tilde{\phi}_{0}^{\prime}\right|^{2}\mathrm{d}\xi & =\int_{-\infty}^{\infty}\left|\partial_{\xi}\tilde{\phi}_{0}^{\prime}\right|\sqrt{2F\left(\tilde{\phi}_{0}\right)}\mathrm{d}\xi=\frac{2\sqrt{2}}{3}.\end{aligned}
\]

Next, equations (\ref{model:divu}) and (\ref{model:divJ}) yields
the leading order $\text{\eqref{model:divu}}{}_{I}^{-1}$ and $\text{\eqref{model:divJ}}{}_{I}^{-1}$,
\begin{equation}
\boldsymbol{\nu}\cdot\partial_{\xi}\tilde{\boldsymbol{u}}_{0}=0,\quad\boldsymbol{\nu}\cdot\partial_{\xi}\tilde{\boldsymbol{J}}_{0}=0.\label{eq:inuj}
\end{equation}
Upon integrating the equation with respect to $\xi$ in $\ensuremath{(-\infty,\infty)}$,
using $\partial_{\xi}\boldsymbol{\nu}=0$ and the matching condition
(\ref{match:a}), we obtain
\[
\left[\boldsymbol{u}_{0}\cdot\boldsymbol{\nu}\right]=0,\quad\left[\boldsymbol{J}_{0}\cdot\boldsymbol{\nu}\right]=0.
\]

Then, we turn to the he momentum equation (\ref{model:fluid}). From
the leading order $\eqref{model:fluid}_{I}^{-2}$, we get
\[
\varepsilon^{-2}\partial_{\xi}\left(\eta\left(\tilde{\varphi}_{0}\right)\mathcal{E}\left(\partial_{\xi}\tilde{\boldsymbol{u}}_{0}\otimes\boldsymbol{\nu}\right)\boldsymbol{\nu}\right)=0.
\]

\noindent Invoking with \eqref{eq:inuj}, we have 
\[
2\mathcal{E}\left(\partial_{\xi}\tilde{\boldsymbol{u}}_{0}\otimes\boldsymbol{\nu}\right)\boldsymbol{\nu}=\partial_{\xi}\tilde{\boldsymbol{u}}_{0}.
\]

\noindent Thus, we conclude that
\[
\partial_{\xi}\left(\eta\left(\tilde{\varphi}_{0}(\xi)\right)\left(\ensuremath{\partial_{\xi}\tilde{\boldsymbol{u}}_{0}}\right)\right)=0.
\]

\noindent Integrating from $-\infty$ to $\xi$ with $|\xi|<\infty,$
using the matching condition (\ref{match:b}) and the positivity of
$\eta(\cdot),$ it gives 
\[
\ensuremath{\partial_{\xi}\tilde{\boldsymbol{u}}_{0}}=\mathbf{0}.
\]
Once more integrating with respect to $\xi$ in $\ensuremath{(-\infty,\infty)}$
and using the matching condition (\ref{match:a}) yields 
\[
\left[\boldsymbol{u}_{0}\right]=\mathbf{0}.
\]

Now we analyze the equation (\ref{model:poisson}) at leading order
$\eqref{model:poisson}{}_{I}^{-1}$,
\[
\boldsymbol{\nu}\partial_{\xi}\tilde{\phi}_{\epsilon}=\boldsymbol{0}.
\]

\noindent Integrating and applying the matching condition (\ref{match:a})
to $\tilde{\phi}_{\epsilon}$ leads to
\[
\left[\phi_{0}\right]=0.
\]

Finally, we consider (\ref{model:ch}). We distinguish again the two
cases for the mobilities:
\begin{itemize}
\item \textbf{Case I} : At order $\epsilon^{-2}$, we obtain from $\eqref{model:ch}{}_{I}^{-2}$
,
\[
\partial_{\xi\xi}\tilde{\mu}_{0}=0.
\]
\item \textbf{Case II} : We obtain from $\eqref{model:ch}{}_{I}^{-1}$ that
\[
\partial_{t}d\partial_{\xi}\tilde{\varphi}_{0}+\partial_{\xi}\left(\tilde{\varphi}_{0}\tilde{\boldsymbol{u}}_{0}\right)\cdot\boldsymbol{\nu}=m_{0}\partial_{\xi\xi}\tilde{\mu}_{0}
\]
Integrating this identity with respect to $\xi$ and using the matching
condition (\ref{match:b}) give
\[
\partial_{t}d+\boldsymbol{u}_{0}\cdot\boldsymbol{\nu}=0
\]
This implies that the normal velocity of the interface $\Gamma$ is
given by 
\[
\mathcal{V}=\boldsymbol{u}_{0}\cdot\boldsymbol{\nu}
\]
In particular, we obtain that 
\[
\partial_{\xi\xi}\tilde{\mu}_{0}=0.
\]
\end{itemize}
Thus, for the \textbf{Case I} and \textbf{Case II}, similar to the
arguments for $\tilde{\boldsymbol{u}}_{0}$, we obtain $\mu_{0}$
is independent of $\xi$, and
\[
\left[\mu_{0}\right]=0.
\]

\begin{itemize}
\item \textbf{Case III} : From $\eqref{model:ch}{}_{I}^{-2}$, we get 
\[
\partial_{t}d\partial_{\xi}\tilde{\varphi}_{0}+\partial_{\xi}\left(\tilde{\varphi}_{0}\tilde{\boldsymbol{u}}_{0}\right)\cdot\boldsymbol{\nu}=\partial_{\xi}\left(m_{0}\left(1-\tilde{\varphi}_{0}^{2}\right)\partial_{\xi}\tilde{\mu}_{0}\right).
\]
The similar arguments as above yields
\[
\mathcal{V}=\boldsymbol{u}_{0}\cdot\boldsymbol{\nu}
\]
and
\begin{align*}
\partial_{\xi}\left(m_{0}\left(1-\tilde{\varphi}_{0}^{2}\right)\partial_{\xi}\tilde{\mu}_{0}\right) & =0.
\end{align*}
Integrating it from $-\infty$ to $\xi$ with $|\xi|<\infty$ and
using the matching condition (\ref{match:b}) yields 
\[
m_{0}\left(1-\tilde{\varphi}_{0}^{2}\right)\partial_{\xi}\tilde{\mu}_{0}=0.
\]
 Note that $\left|\tilde{\varphi}_{0}\right|<1$ for $|\xi|<\infty$,
this implies that 
\[
\partial_{\xi}\tilde{\mu}_{0}=0.
\]
Thus, we conclude that $\mu_{0}$ is independent of $\xi$, and
\[
\left[\mu_{0}\right]=0.
\]
\end{itemize}

\subsubsection{Inner Expansion to higher order}

We will now expand the equations in the inner regions to the next
highest order for obtaining contributions of the diffusive fluxes
in the interface, and the momentum balance in the sharp interface
limit.

First of all, from $\text{\eqref{model:chc}}{}_{I}^{0}$, we obtain
\[
\tilde{\mu}_{0}=\lambda f^{\prime}\left(\tilde{\varphi}_{0}\right)\tilde{\varphi}_{1}-\lambda\partial_{\xi\xi}\tilde{\varphi}_{1}+\lambda\partial_{\xi}\tilde{\varphi}_{0}\kappa
\]

\noindent Multiplying by $\partial_{\xi}\tilde{\phi}_{0}$ and integrating
from $-\infty$ to $+\infty$, and applying the fact $\tilde{\mu}_{0}$
is independent of $\xi$ and the matching condition (\ref{match:a})
yields 
\begin{align*}
2\mu_{0} & =\int_{-\infty}^{\infty}\tilde{\mu}_{0}\partial_{\xi}\tilde{\varphi}_{0}\mathrm{d}\xi=\int_{-\infty}^{\infty}\lambda\partial_{\xi}\left(f\left(\tilde{\varphi}_{0}\right)\right)\tilde{\varphi}_{1}-\lambda\partial_{\xi\xi}\tilde{\varphi}_{1}\partial_{\xi}\tilde{\varphi}_{0}+\lambda\kappa\left|\partial_{\xi}\tilde{\varphi}_{0}\right|^{2}\mathrm{d}\xi\\
 & =\int_{-\infty}^{\infty}\lambda\partial_{\xi}\left(f\left(\tilde{\varphi}_{0}\right)\right)\tilde{\varphi}_{1}-\lambda\partial_{\xi\xi}\tilde{\varphi}_{1}\partial_{\xi}\tilde{\varphi}_{0}+\lambda\kappa\iota
\end{align*}
 Using the matching conditions (\ref{match:a}) and (\ref{match:b}),
(\ref{inner:ch}) and $f(\pm1)=0,$ integration by parts gives 
\[
\begin{aligned}\int_{-\infty}^{\infty}\partial_{\xi}\left(f\left(\tilde{\varphi}_{0}\right)\right)\tilde{\varphi}_{1}-\partial_{\xi\xi}\tilde{\varphi}_{1}\partial_{\xi}\tilde{\varphi}_{0}\mathrm{d}\xi= & \left[f\left(\tilde{\varphi}_{0}\right)\tilde{\varphi}_{1}-\partial_{\xi}\tilde{\varphi}_{1}\partial_{\xi}\tilde{\varphi}_{0}\right]_{-\infty}^{+\infty}\\
 & -\int_{-\infty}^{\infty}\partial_{\xi}\tilde{\varphi}_{1}\left(f\left(\tilde{\varphi}_{0}\right)-\partial_{\xi\xi}\tilde{\varphi}_{0}\right)\mathrm{d}\xi=0
\end{aligned}
\]

\noindent Introducing the scaled surface tension between the two phases
$\hat{\lambda}=\lambda\iota$, Collecting these equalities gives 
\begin{equation}
2\mu_{0}=\hat{\lambda}\kappa\label{inner:gibbs}
\end{equation}

\noindent This is a solvability condition for $\tilde{\varphi}_{1}$,
the so-called Gibbs-Thomas equation \citep{Abels2012,Ebenbeck2020}.

Now we analyze (\ref{model:ch}) only for the mobilities (\ref{eq:mob})
for the \textbf{Case I}. For the \textbf{Case II} and \textbf{Case
III}, the mobility does not contribute to the sharp interface limit.
Using $\partial_{\xi}\tilde{\mu}_{0}=0$ and from $\text{\eqref{model:ch}}_{I}^{-1}$,
we obtain 
\[
\left(-\mathcal{V}+\tilde{\boldsymbol{u}}_{0}\cdot\boldsymbol{\nu}\right)\partial_{\xi}\tilde{\varphi}_{0}=m_{0}\partial_{\xi\xi}\tilde{\mu}_{1}
\]
 Integrating with respect to $\xi$ from $-\infty$ to $\infty,$
using the matching condition (\ref{match:c}) yields 
\[
2\left(-\mathcal{V}+\boldsymbol{u}_{0}\cdot\boldsymbol{\nu}\right)=m_{0}\left[\nabla\mu_{0}\cdot\boldsymbol{\nu}\right]
\]

\noindent Finally, we analyze (\ref{model:fluid}). We obtain the
equation (\ref{model:fluid}) at order $\epsilon^{-1}$, $\text{\eqref{model:fluid}}{}_{I}^{-1}$
,
\begin{equation}
\partial_{\xi}\left(2\eta\left(\tilde{\varphi}_{0}\right)\mathcal{E}\left(\partial_{\xi}\tilde{\boldsymbol{u}}_{1}\otimes\boldsymbol{\nu}\right)\boldsymbol{\nu}\right)+\partial_{\xi}\left(2\eta\left(\tilde{\varphi}_{0}\right)\mathcal{E}\left(\nabla\tilde{\boldsymbol{u}}_{0}\right)\boldsymbol{\nu}\right)-\partial_{\xi}\tilde{p}_{0}\boldsymbol{\nu}+\mu_{0}\partial_{\xi}\tilde{\varphi}_{0}\boldsymbol{\nu}=0.\label{eq:inu}
\end{equation}

\noindent Using the match condition (\ref{match:c}) to $\tilde{\boldsymbol{u}}_{1}$,
we have 
\[
\underset{\xi\rightarrow\pm\infty}{\lim}\left(\partial_{\xi}\tilde{\boldsymbol{u}}_{1}\otimes\boldsymbol{\nu}+\nabla\tilde{\boldsymbol{u}}_{0}\right)=\nabla\boldsymbol{u}_{0}.
\]

\noindent Integrating (\ref{eq:inu}) with respect to $\xi$, using
(\ref{inner:gibbs}) and after matching gives

\[
\left[2\eta D(\boldsymbol{u}_{0})\boldsymbol{\nu}-p\boldsymbol{\nu}\right]=\hat{\lambda}\kappa\boldsymbol{\nu}.
\]
It should be noted that this is the classical jump condition for the
stress \citep{Shen2012}.

\subsection{Sharp interface limits and energy inequalities}

First of all, we summarize the sharp interface limits for (\ref{model:chimhd})
with the different mobilities. In the sharp interface limits, not
only do we need to seek for the functions $\left(\boldsymbol{u},p,\boldsymbol{J},\phi,\mu\right)$
for \textbf{Case I} or $\left(\boldsymbol{u},p,\boldsymbol{J},\phi\right)$
for \textbf{Cases II} and \textbf{III}, but we also need to look for
a smoothly evolving hyper-surface $\Gamma$. Specifically,
\begin{itemize}
\item \textbf{Case I} : The six-tuple $\left(\Gamma,\boldsymbol{u},p,\boldsymbol{J},\phi,\mu\right)$
satisfies
\end{itemize}
\begin{equation}
\begin{aligned}\partial_{t}\boldsymbol{u}+\left(\boldsymbol{u}\cdot\nabla\right)\boldsymbol{u}-2\eta_{i}{\rm div}\left(D\left(\boldsymbol{u}\right)\right)+\nabla p-\boldsymbol{J}\times\boldsymbol{B} & =\boldsymbol{0}\quad\text{in }\Omega\times(0,\mathrm{T}],\\
{\rm div}\boldsymbol{u} & =0\quad\text{in }\Omega\times(0,\mathrm{T}],\\
\sigma_{i}^{-1}\boldsymbol{J}+\nabla\phi-\boldsymbol{u}\times\boldsymbol{B} & =\boldsymbol{0}\quad\text{in }\Omega\times(0,\mathrm{T}],\\
\mathrm{div}\boldsymbol{J} & =0\quad\text{in }\Omega\times(0,\mathrm{T}],\\
\Delta\mu & =0\quad\text{in }\Omega\times(0,\mathrm{T}],\\
\left[2\eta D(\boldsymbol{u})\boldsymbol{\nu}-p\boldsymbol{\nu}\right]=\hat{\lambda}\kappa\boldsymbol{\nu},\quad\left[\boldsymbol{u}\right] & =0\quad\text{on }\Gamma\times(0,\mathrm{T}],\\
\left[\boldsymbol{J}\cdot\boldsymbol{\nu}\right]=0,\quad\left[\phi\right] & =0\quad\text{on }\Gamma\times(0,\mathrm{T}],\\
2\left(-\mathcal{V}+\boldsymbol{u}\cdot\boldsymbol{\nu}\right)=m_{0}\left[\nabla\mu\right]\cdot\boldsymbol{\nu},\quad2\mu & =\hat{\lambda}\kappa\quad\text{on }\Gamma\times(0,\mathrm{T}],
\end{aligned}
\label{sharp:chmhdI}
\end{equation}
\noindent together with the boundary and initial conditions
\begin{align}
\boldsymbol{u}(0) & =\boldsymbol{u}^{0}\quad\text{ in }\Omega,\\
\boldsymbol{u}=0,\quad\boldsymbol{J}\cdot\boldsymbol{n}=0,\quad\partial_{\boldsymbol{n}}\mu & =0\text{ on }\Sigma\times(0,\mathrm{T}],\\
\Gamma\left(0\right) & =\Gamma_{0}
\end{align}
\begin{itemize}
\item \textbf{Case II} and \textbf{III }: The five-tuple $\left(\Gamma,\boldsymbol{u},p,\boldsymbol{J},\phi\right)$
satisfies
\end{itemize}
\begin{equation}
\begin{aligned}\partial_{t}\boldsymbol{u}+\left(\boldsymbol{u}\cdot\nabla\right)\boldsymbol{u}-2\eta_{i}{\rm div}\left(D\left(\boldsymbol{u}\right)\right)+\nabla p-\boldsymbol{J}\times\boldsymbol{B} & =\boldsymbol{0}\quad\text{in }\Omega\times(0,\mathrm{T}],\\
{\rm div}\boldsymbol{u} & =0\quad\text{in }\Omega\times(0,\mathrm{T}],\\
\sigma_{i}^{-1}\boldsymbol{J}+\nabla\phi-\boldsymbol{u}\times\boldsymbol{B} & =\boldsymbol{0}\quad\text{in }\Omega\times(0,\mathrm{T}],\\
\mathrm{div}\boldsymbol{J} & =0\quad\text{in }\Omega\times(0,\mathrm{T}],\\
\left[2\eta D(\boldsymbol{u})\boldsymbol{\nu}-p\boldsymbol{\nu}\right]=\hat{\lambda}\kappa\boldsymbol{\nu},\quad\left[\boldsymbol{u}\right] & =0\quad\text{on }\Gamma\times(0,\mathrm{T}],\\
\left[\boldsymbol{J}\cdot\boldsymbol{\nu}\right]=0,\quad\left[\phi\right] & =0\quad\text{on }\Gamma\times(0,\mathrm{T}],\\
\mathcal{V} & =\boldsymbol{u}\cdot\boldsymbol{\nu}\quad\text{on }\Gamma\times(0,\mathrm{T}],
\end{aligned}
\label{sharp:chmhdII}
\end{equation}
\noindent together with the boundary and initial conditions 
\begin{align}
\boldsymbol{u}(0) & =\boldsymbol{u}^{0}\quad\text{ in }\Omega,\\
\boldsymbol{u}=0,\quad\boldsymbol{J}\cdot\boldsymbol{n} & =0\text{ on }\Sigma\times(0,\mathrm{T}],\\
\Gamma\left(0\right) & =\Gamma_{0}
\end{align}

It is remarkable that \textbf{Case I} has distinctive features on
the bulk equations and interface conditions, compared with \textbf{Case
II} and \textbf{III}. 
\begin{itemize}
\item Bulk equations: For the case of $M\left(\varphi\right)=\epsilon m_{0}$
tends to zero or $M\left(\varphi\right)=m_{0}\left(1-\varphi\right)_{+}^{2}$
degenerates in the bulk, the equations are standard inductionless
MHD equations. While in the case of a constant mobility $M\left(\varphi\right)=m_{0}$,
a harmonic equation for the chemical potential are incorporated in
addition due to the diffusion of mass.
\item Interface conditions: The three cases under consideration gives the
same interface conditions for hydrodynamics and electrostatics, the
velocity is continuous and the stress tensor fulfills the Yong-Laplace
law, the current density is normal-continuous and the electric potential
is continuous. For \textbf{Case I}, the interface condition for the
chemical potential is an identity, which indicate that $\mu$ is a
constant related the surface tension coefficient and the mean curvature
of the interface. A crucial observation is that the evolution of the
interface is quite different. In \textbf{Case II} and \textbf{III},
we get the usual kinematic condition that the interface is only transported
by the flow of the surrounding fluids. Roughly speaking, the fluid
interface evolves as a scaled mean curvature flow \citep{Du2020}.
However, in \textbf{Case I}, the interface is no longer material and
the interface condition is of Stefan type condition. Both the velocity
of the fluid and the jump of the flux of the chemical potential on
the interface contribute the normal velocity of the interface. 
\end{itemize}
Finally, we derive a formal energy law fo the sharp interface model.
\begin{thm}
\label{thm:engysharp}Assume the solution to the sharp interface problems
and the evolving hyper-surface $\Gamma\subset\Omega$ are sufficiently
smooth, then it holds
\begin{equation}
\frac{{\rm d}}{{\rm d}t}\left[\int_{\Omega}\left(\frac{1}{2}|\boldsymbol{u}|^{2}\right)+\int_{\Gamma}\hat{\lambda}\right]=-\mathcal{D}\leq0\label{eq:engylawssharp}
\end{equation}
 where the dissipation quantity $\mathcal{D}$ is defined by 
\begin{align*}
\textrm{\text{ \textbf{Case I} }}: & \mathcal{D}=\int_{\Omega}2\eta\left|D\left(\boldsymbol{u}\right)\right|^{2}+\int_{\Omega}\sigma^{-1}\left|\boldsymbol{J}\right|^{2}+\int_{\Omega}m_{0}|\nabla\mu|^{2},\\
\textrm{\text{ \textbf{Case II} and \textbf{III} }}: & \mathcal{D}=\int_{\Omega}2\eta\left|D\left(\boldsymbol{u}\right)\right|^{2}+\int_{\Omega}\sigma^{-1}\left|\boldsymbol{J}\right|^{2}.
\end{align*}
\end{thm}
\begin{proof}
We begin with calculating the rate of change of the kinetic energy.
Using the first two equations of (\ref{sharp:chmhdI}) and integration
by parts on $\Omega_{1}$ and $\Omega_{1}$, we obtain
\begin{align}
\frac{{\rm d}}{{\rm d}t}\left(\frac{1}{2}\int_{\Omega}|\boldsymbol{u}|^{2}\right) & =\int_{\Omega_{+}\cup\Omega_{-}}\boldsymbol{u}\cdot\left(2\eta_{i}{\rm div}\left(D\left(\boldsymbol{u}\right)\right)-\nabla p+\boldsymbol{J}\times\boldsymbol{B}\right)\nonumber \\
 & =-\int_{\Omega}2\eta\left|D\left(\boldsymbol{u}\right)\right|^{2}+\int_{\Omega}\boldsymbol{u}\cdot\left(\boldsymbol{J}\times\boldsymbol{B}\right)\nonumber \\
 & \quad\ensuremath{+\int_{\Gamma}\boldsymbol{u}\cdot\left(\left[p\right]\boldsymbol{\nu}-2\left[\eta D\left(\boldsymbol{u}\right)\right]\boldsymbol{\nu}\right)}\nonumber \\
 & =-\int_{\Omega}2\eta\left|D\left(\boldsymbol{u}\right)\right|^{2}+\int_{\Omega}\boldsymbol{u}\cdot\left(\boldsymbol{J}\times\boldsymbol{B}\right)\nonumber \\
 & \quad+\int_{\Gamma}\hat{\lambda}\kappa\boldsymbol{u}\cdot\boldsymbol{\nu}.\label{eni:u}
\end{align}
In the procedure of derivation in second equality, we have used the
interface condition that $\boldsymbol{u}$ is continuous across $\Gamma$.
The interface condition for $p$ is applied to get the second equality.
To handle the second term, taking the inner product of third equations
in (\ref{sharp:chmhdI}) with $\boldsymbol{J}$, integration
by parts on $\Omega_{+}$ and $\Omega_{-}$ and using the interface
condition for $\boldsymbol{J}$ and $\phi$, we arrive at 
\begin{align}
\int_{\Omega}\sigma^{-1}\boldsymbol{J}\cdot\boldsymbol{J} & =\int_{\Omega}\left(\boldsymbol{u}\times\boldsymbol{B}\right)\cdot\boldsymbol{J}+\int_{\Gamma}\phi\cdot\left(\left[\boldsymbol{J}\cdot\boldsymbol{\nu}\right]\right)\nonumber \\
 & =-\int_{\Omega}\boldsymbol{u}\cdot\left(\boldsymbol{J}\times\boldsymbol{B}\right).\label{eni:J}
\end{align}
Plugging (\ref{eni:J}) into (\ref{eni:u}) yields, 
\begin{equation}
\frac{{\rm d}}{{\rm d}t}\left(\frac{1}{2}\int_{\Omega}|\boldsymbol{u}|^{2}\right)=-\int_{\Omega}2\eta\left|D\left(\boldsymbol{u}\right)\right|^{2}-\int_{\Omega}\sigma^{-1}\left|\boldsymbol{J}\right|^{2}+\int_{\Gamma}\hat{\lambda}\kappa\boldsymbol{u}\cdot\boldsymbol{\nu}.\label{eni:uJ}
\end{equation}
Next, we compute the the rate of change of interfacial energy by using
the transport identity \citep{Deckelnick2005,Abels2012},
\begin{align}
\frac{{\rm d}}{{\rm d}t}\left(\int_{\Gamma}\hat{\lambda}\right) & =-\int_{\Gamma}\hat{\lambda}\kappa\mathcal{V}.\label{eni:lambda}
\end{align}
Combing (\ref{eni:uJ}) and \ref{eni:lambda}, it gives 
\begin{equation}
\frac{{\rm d}}{{\rm d}t}\left(\frac{1}{2}\int_{\Omega}|\boldsymbol{u}|^{2}+\int_{\Gamma}\hat{\lambda}\right)=-\int_{\Omega}2\eta\left|D\left(\boldsymbol{u}\right)\right|^{2}-\int_{\Omega}\sigma^{-1}\left|\boldsymbol{J}\right|^{2}+\int_{\Gamma}\hat{\lambda}\kappa\left(\boldsymbol{u}\cdot\boldsymbol{\nu}-\mathcal{V}\right).\label{eni:uJlambda}
\end{equation}
For the \textbf{Case II} and \textbf{III}, applying the interface
condition that $\mathcal{V}=\boldsymbol{u}\cdot\boldsymbol{\nu}$
to (\ref{eni:uJlambda}) yields the desired result. 

In the \textbf{Case I}, with the help of the interface condition for
$\mu$, we rewrite the last term as 
\begin{equation}
\int_{\Gamma}\lambda\kappa\left(\boldsymbol{u}\cdot\boldsymbol{\nu}-\mathcal{V}\right)=\int_{\Gamma}2\mu\left(\boldsymbol{u}\cdot\boldsymbol{\nu}-\mathcal{V}\right)=\int_{\Gamma}m_{0}\mu\left[\nabla\mu\right]\cdot\boldsymbol{\nu}.\label{eni:mu}
\end{equation}
Invoking the equation for $\mu$ and integration by parts on $\Omega_{+}$
and $\Omega_{-}$, we obtain
\[
\int_{\Gamma}m_{0}\mu\left[\nabla\mu\right]\cdot\boldsymbol{\nu}=\ensuremath{-\int_{\Omega}\nabla\cdot\left(\mu m_{0}\nabla\mu\right)=-\int_{\Omega}m_{0}|\nabla\mu|^{2}}.
\]
Using this equation, we can simplify the equation (\ref{eni:mu})
to 
\begin{equation}
\int_{\Gamma}\hat{\lambda}\kappa\left(\boldsymbol{u}\cdot\boldsymbol{\nu}-\mathcal{V}\right)=-\int_{\Omega}m_{0}|\nabla\mu|^{2}.\label{eni:kappa}
\end{equation}
Substituting \ref{eni:kappa} into (\ref{eni:uJlambda}) gives the
result. The proof completes.
\end{proof}
Compared Theorem \ref{thm:moengy} with Theorem \ref{thm:engysharp},
we see that the sharp interface limit of the energy identity (\ref{eq:engylawssharp})
is formally identical to the one (\ref{eq:moengylaw}) of the diffuse
interface model in \textbf{Case I}. While for \textbf{Cases II} and
\textbf{III}, we observe that the diffusion of mass in the the energy
identity (\ref{eq:engylawssharp}) of diffuse interface model is not
present in the dissipation quantity $\mathcal{D}$ of the sharp interface
problems. From the above analysis, we would like to remark that for
different choices of the mobilities leads to the different sharp interface
problem. In real applications, one could pick out the mobility based
on ones' own purpose.
\begin{rem}
\label{rem:GeneralPoten}It is trivial to extend the cases which satisfy
the hypothesis that $F(\varphi)$ is a double-well potential with
two equal minima at $\pm1$. With these types $F(\varphi)$, only
the expressions for $\tilde{\varphi}_{0}(\xi)$ and $\tau$ will not
be given explicitly. 
\end{rem}
\begin{rem}
We would like to remark that we did not consider the convergence rate
of the sharp-interface limits or prove that weak solutions tend to
varifold solutions of a corresponding sharp interface model. This
is an important and interesting project, and we left it in future
work.
\end{rem}

\section{Numerical Experiments\label{sec:Numer}}

In this section, we numerically verify the convergence behavior of
the diffuse interface model for different mobility with respect to
$\epsilon$. For that end, we consider a flow in a square domain $\Omega=(0,1)^{2}$.
We set the external magnetic field to be $\boldsymbol{B}=\left(0,0,1\right)$.
The initial velocity field has the form of a large vortex (see Fig.
\ref{fig:InitVel}),
$\boldsymbol{u}_{0}=\left(x^{2}(1-x)^{2}y(1-y)(1-2y),-x(1-x)(1-2x)y^{2}(1-y)^{2}\right).$

\begin{figure}
\begin{centering}
\includegraphics[scale=0.25]{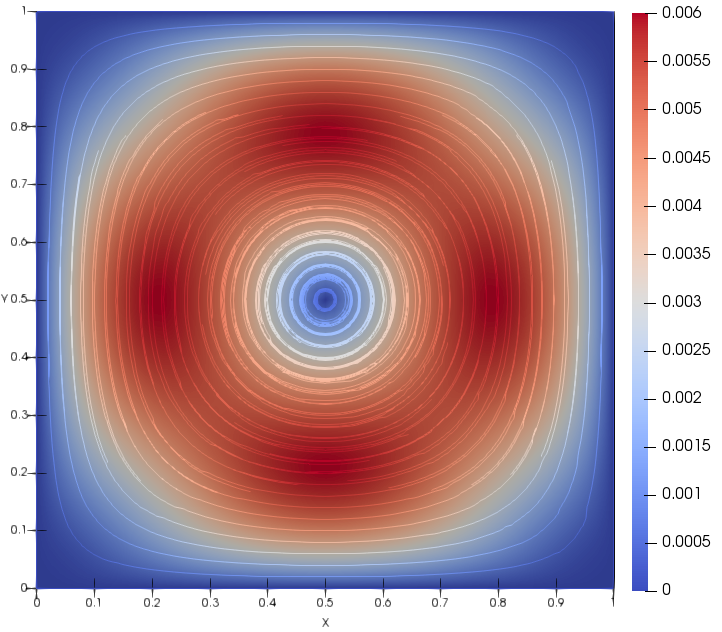}
\par\end{centering}
\caption{The streamlines and magnitudes of initial velocity field\label{fig:InitVel}}
\end{figure}

In the implementation, we employ a decoupled and, linear, and energy
stable finite element method recently proposed by \cite{Wang2021}.
For the reader's convenience, the scheme is list in the appendix.
Since there is no available analytical solutions of the sharp interface
limits, we regard the numerical solution with the interfacial width
$\epsilon=0.01$ as the reference solution for the comparison. To
verify the convergence of the solution with respect to $\epsilon$,
we vary $\epsilon$ gradually from $0.1$ to $0.0125$. 

First of all, we consider the dynamics of a rounded square. To that
end, we initiate the phase field as, 
\[
\varphi_{0}=\tanh\left(\frac{\left\Vert x-x_{c}\right\Vert _{4}-0.3}{\sqrt{2}\varepsilon}\right).
\]

\noindent Here $x_{c}=\left(0.5,0.5\right)$ is the center of the
rounded square and $\left\Vert x-x_{c}\right\Vert _{4}$ is the $l_{4}$-distance
between the points $x$ and $x_{c}$. The initial profile of the phase
field is shown in Fig. \ref{fig:Initphi}. In the figure, the red
color stands for the fluid 1, the blue color stands for the fluid
2$,$ and the black curve represents the zero-level set of reference
phase function. The physical parameters are taken as $\eta=\sigma=1,$$m_{0}=1$,$\gamma=0.1$
in this first test case.

\begin{figure}
\begin{centering}
\begin{tabular}{cccc}
\includegraphics[scale=0.15]{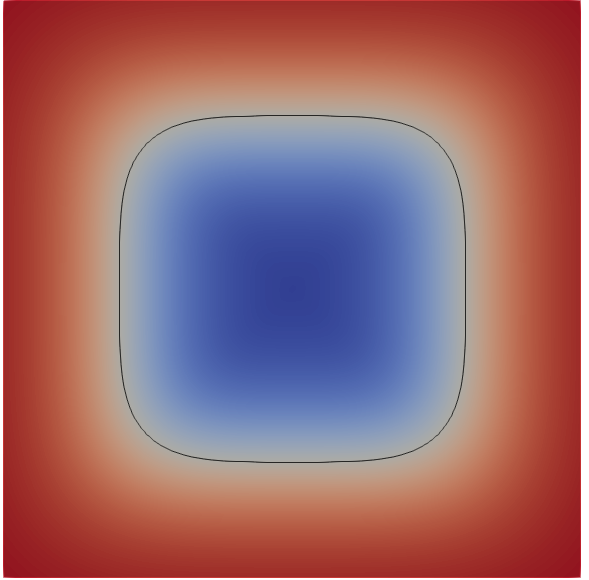} & \includegraphics[scale=0.15]{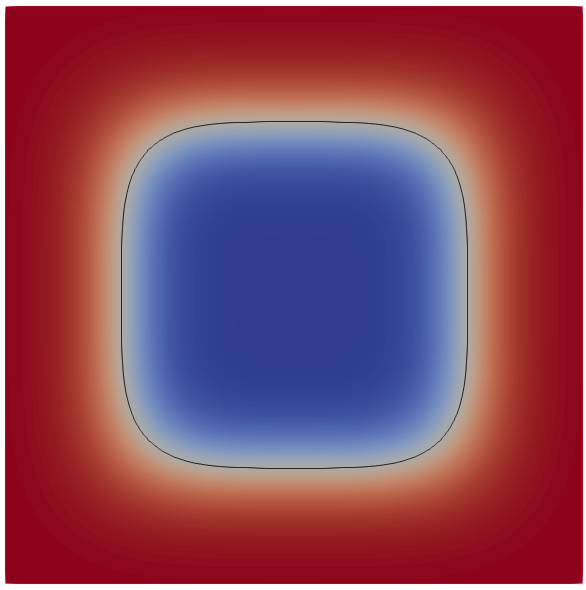} & \includegraphics[scale=0.15]{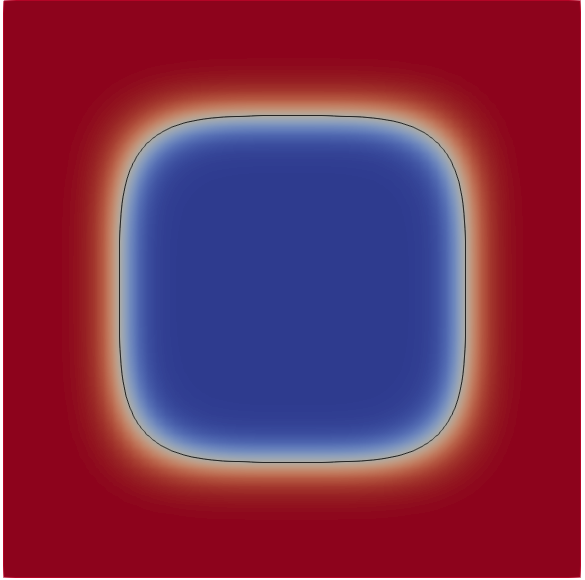} & \includegraphics[scale=0.15]{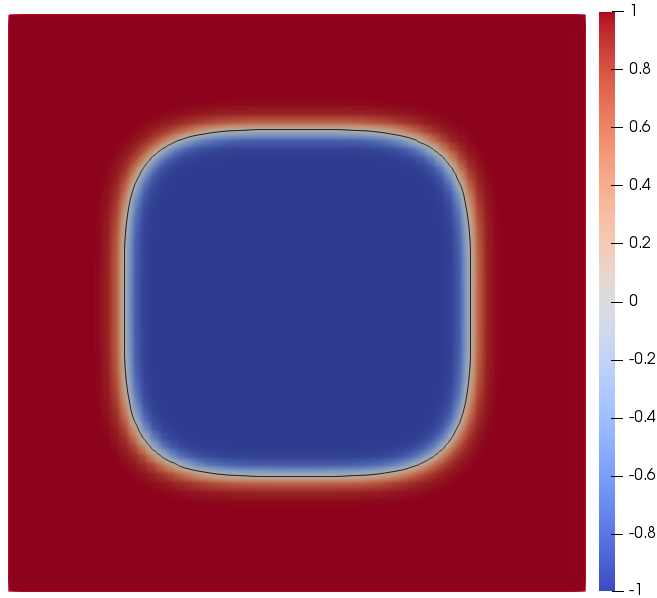}\tabularnewline
\end{tabular}
\par\end{centering}
\caption{Initial initial profile of the phase field for $\epsilon=0.1,0.05,0.025,0.0125$.\label{fig:Initphi}}
\end{figure}

Fig. \ref{fig:Snaps} shows the snapshots of the two-phase interface
at the final time $T=2.0$ in the simulation results for various mobilities
and thickness. From this figure, we see that the isolated rounded
square shape relaxes to a circular shape, due to the effect of surface
tension for all cases. In order to see the convergence of the diffuse-interface
model more clearly, we display the zero-level set of the computed
phase function for different $\epsilon$ in Fig. \ref{fig:Snaps}.
We observe that that as the width $\epsilon$ becomes smaller, the
layers get thinner and converge to the black circle. Thus, it can
be concluded that the diffuse-interface model converges to the sharp
interface limits for all cases. 

\begin{figure}
\begin{centering}
\subfloat[\textbf{Case I}: $M\left(\varphi\right)=m_{0}$ .]{\begin{centering}
\begin{tabular}{cccc}
\includegraphics[scale=0.15]{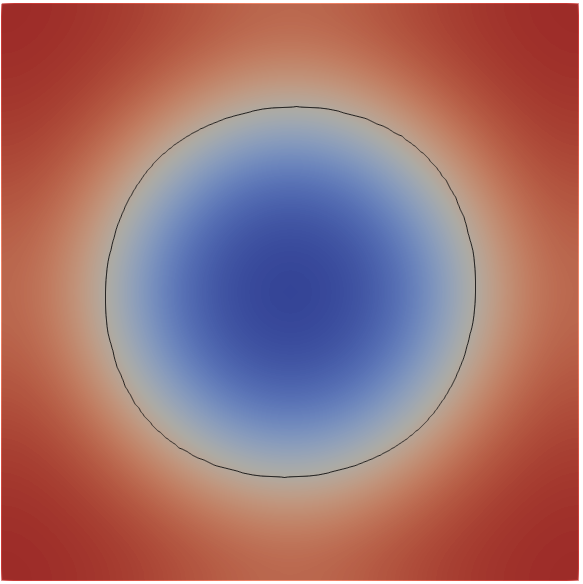} & \includegraphics[scale=0.15]{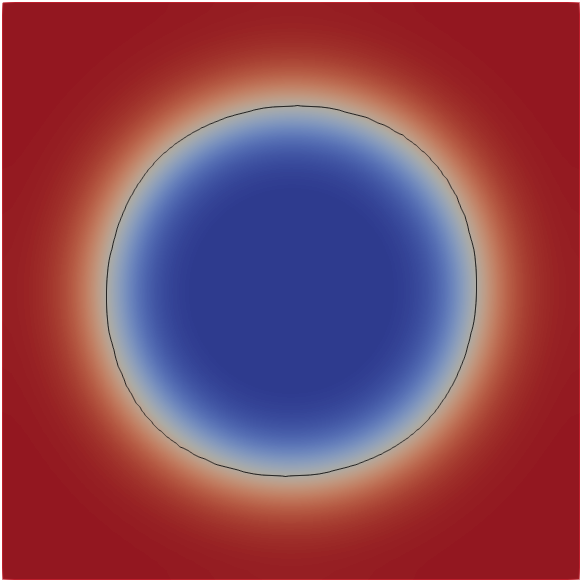} & \includegraphics[scale=0.15]{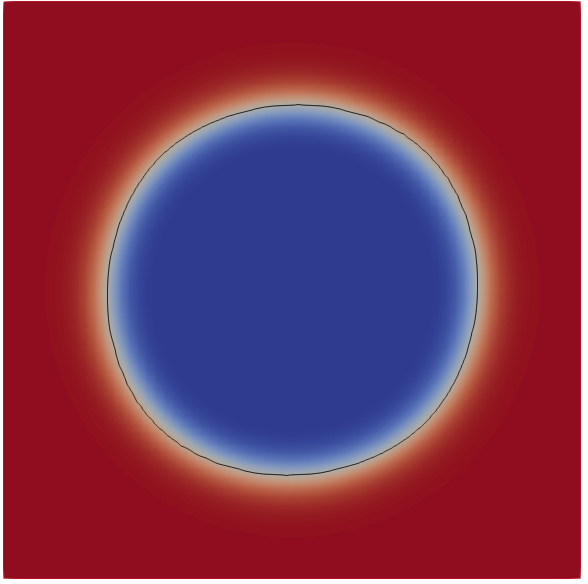} & \includegraphics[scale=0.15]{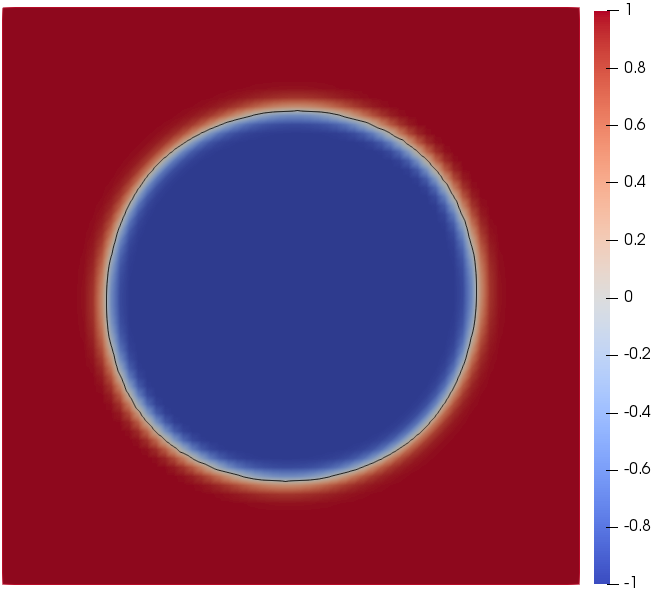}\tabularnewline
\end{tabular}
\par\end{centering}
}
\par\end{centering}
\begin{centering}
\subfloat[\textbf{Case II}: $M\left(\varphi\right)=\epsilon m_{0}$.]{\begin{centering}
\begin{tabular}{cccc}
\includegraphics[scale=0.15]{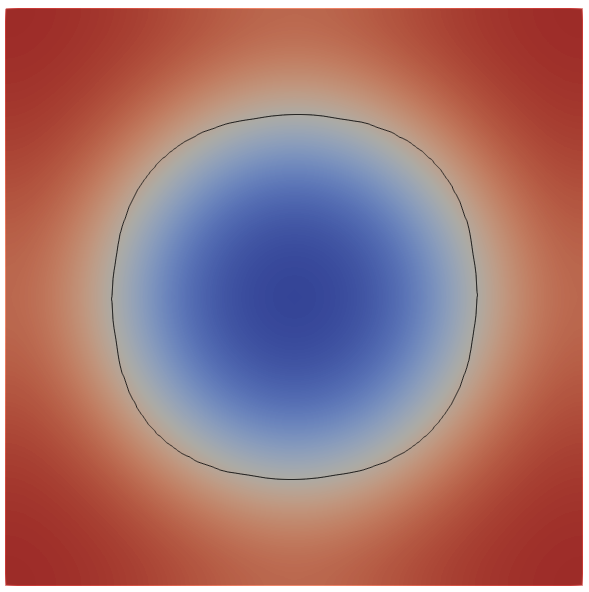} & \includegraphics[scale=0.15]{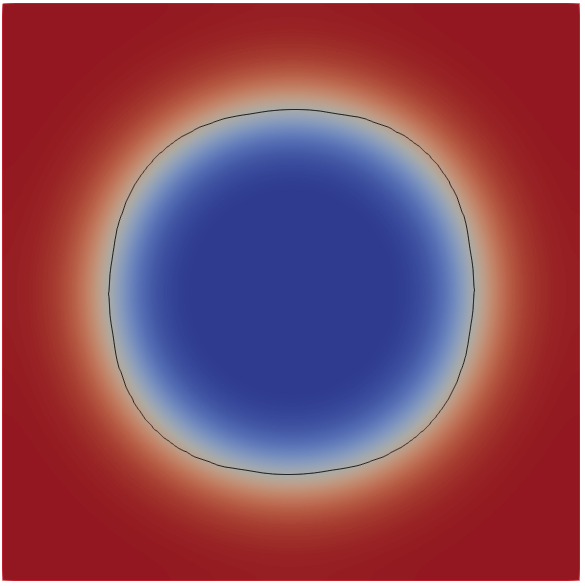} & \includegraphics[scale=0.15]{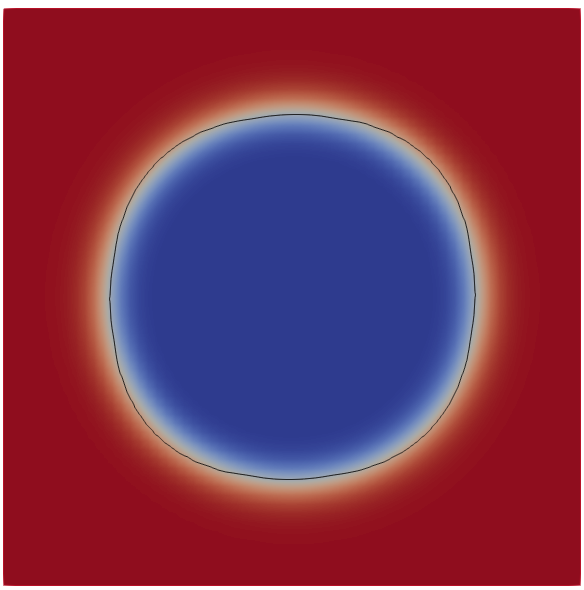} & \includegraphics[scale=0.15]{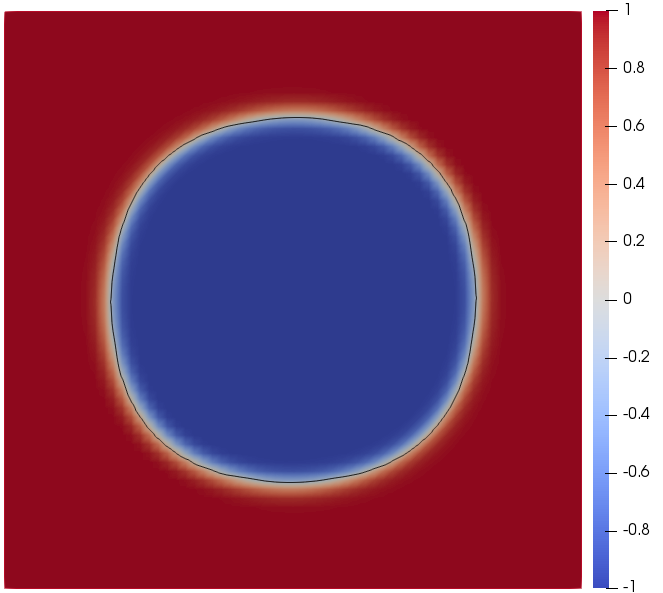}\tabularnewline
\end{tabular}
\par\end{centering}
}
\par\end{centering}
\begin{centering}
\subfloat[\textbf{Case III}: $M\left(\varphi\right)=m_{0}\left(1+\varphi\right)_{+}^{2}$.]{\begin{centering}
\begin{tabular}{cccc}
\includegraphics[scale=0.15]{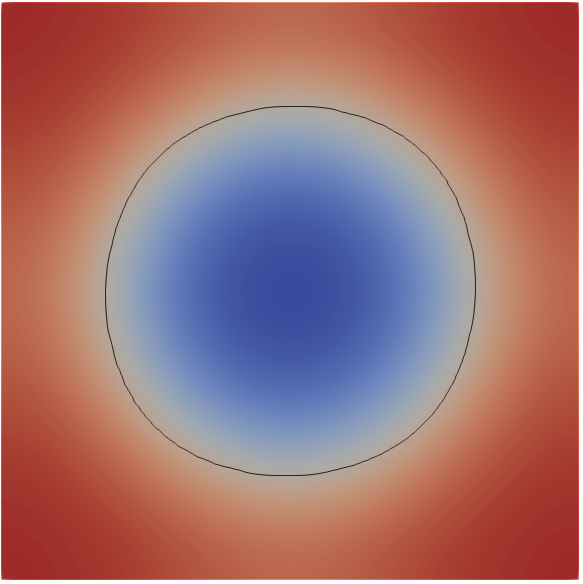} & \includegraphics[scale=0.15]{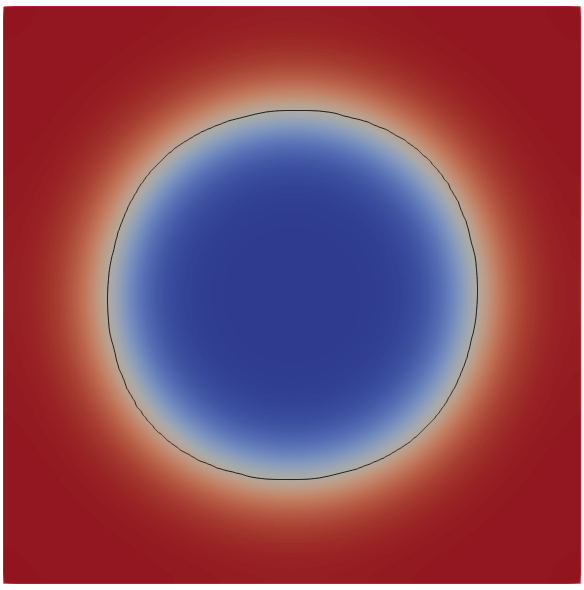} & \includegraphics[scale=0.15]{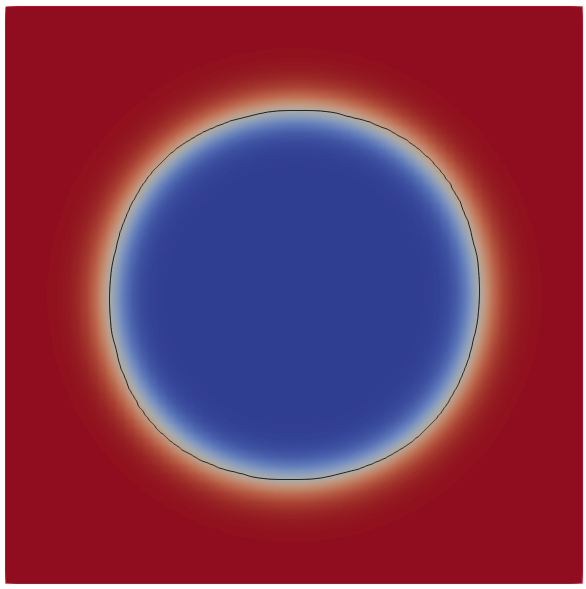} & \includegraphics[scale=0.15]{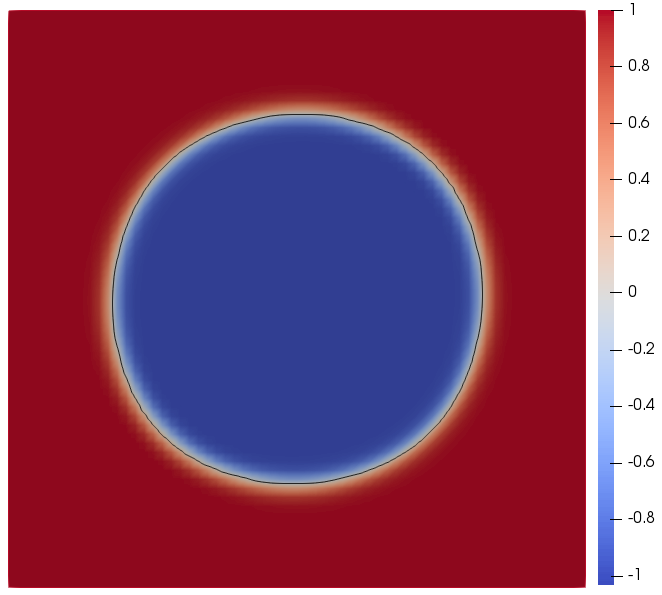}\tabularnewline
\end{tabular}
\par\end{centering}
}
\par\end{centering}
\caption{\label{fig:Snaps}Profile of the phase fields for $\epsilon=0.1,0.05,0.025,0.0125$.(From
left to right.)}
\end{figure}

\begin{figure}
\begin{centering}
\begin{tabular}{ccc}
\includegraphics[scale=0.2]{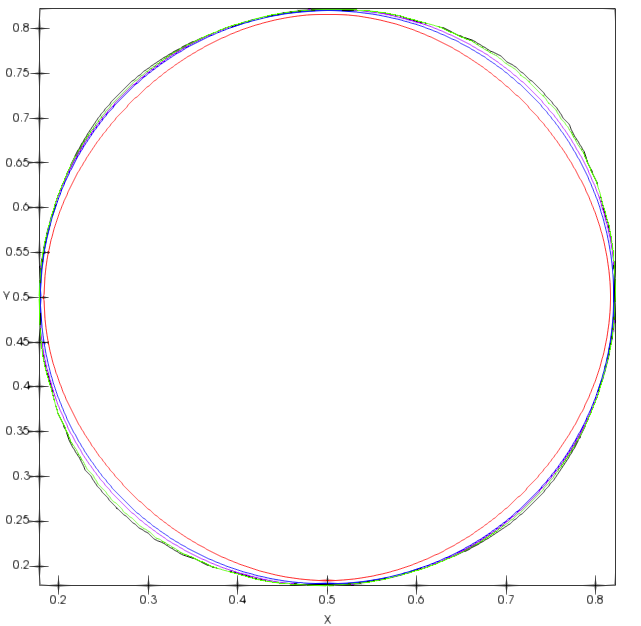} & \includegraphics[scale=0.2]{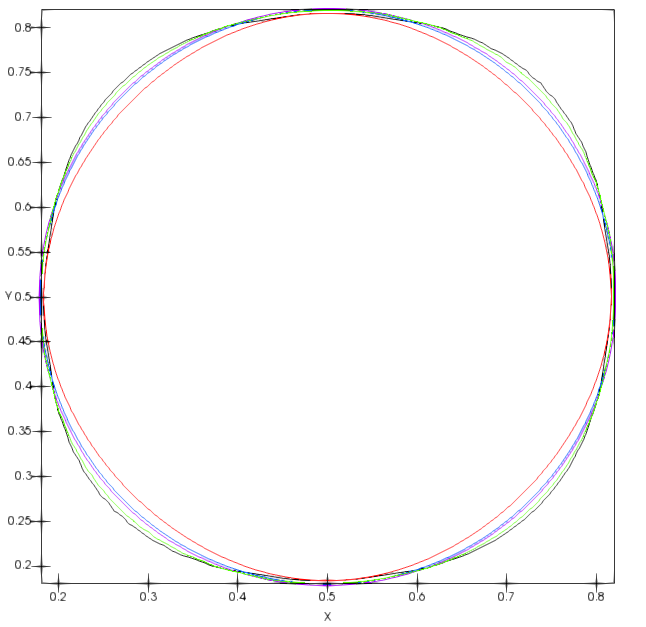} & \includegraphics[scale=0.2]{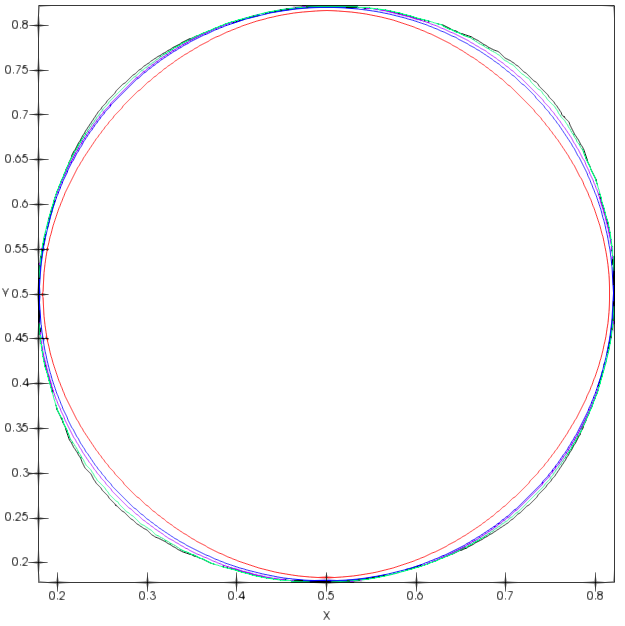}\tabularnewline
\end{tabular}
\par\end{centering}
\caption{The zero-level set of the computed phase function for different $\epsilon$.
Red: $\epsilon=0.1$, Blue: $\epsilon=0.05$, Purple: $\epsilon=0.025$,
Blue: $\epsilon=0.0125$, Black: $\epsilon=0.01$.}
\end{figure}

Next, we consider the evolution of two separately inequal circular
bubbles. The initial profile of phase function $\varphi$ is chosen
to be
\[
\varphi_{0}=1-\tanh\left(\frac{\left\Vert x-x_{o_{1}}\right\Vert -r_{1}}{\sqrt{2}\varepsilon}\right)-\tanh\left(\frac{\left\Vert x-x_{o_{2}}\right\Vert -r_{2}}{\sqrt{2}\varepsilon}\right)
\]
where$\left\Vert x-x_{o_{i}}\right\Vert $ is the Eulerian
distance between the points $x$ and $x_{o_{i}}$, $x_{o_{1}}=\left(0.3,0.5\right)$
and $x_{o_{2}}=\left(0.7,0.5\right)$ , are the center of two bubbles,
$r_{1}=0.1,r_{2}=0.2$ are its radius. Fig. \ref{fig:Initphi-1} displays
the initial profile of the phase field. In this test case, the physical
parameters are selected as $\eta=\sigma=100$,$m_{0}=\gamma=0.01$.

\begin{figure}
\begin{centering}
\begin{tabular}{cccc}
\includegraphics[scale=0.15]{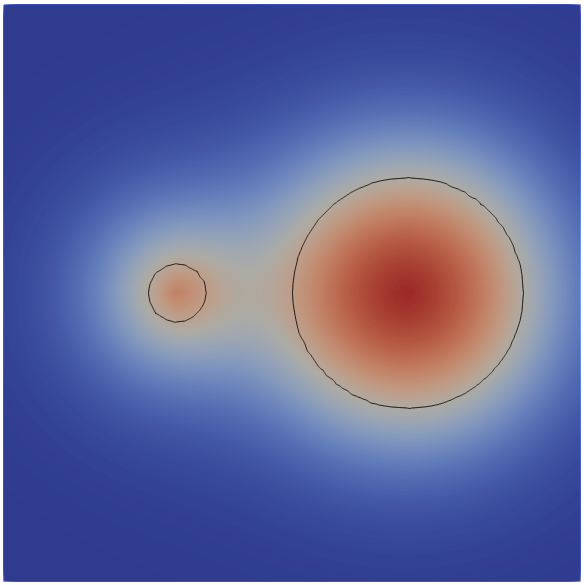} & \includegraphics[scale=0.15]{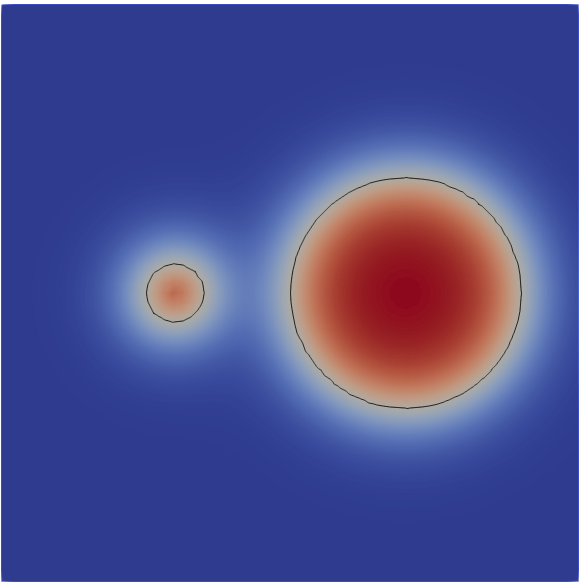} & \includegraphics[scale=0.15]{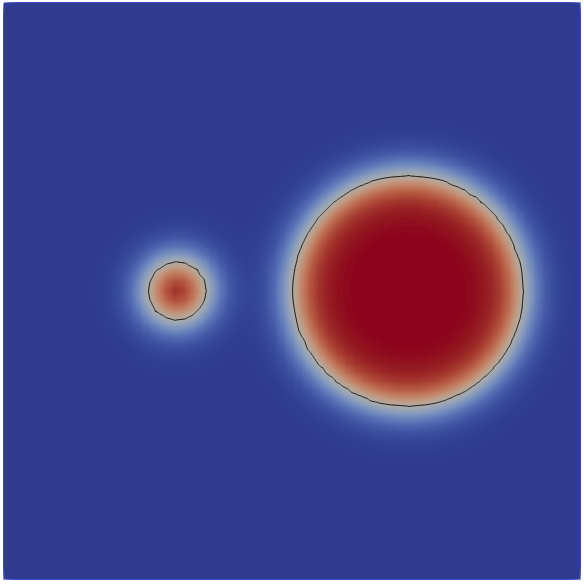} & \includegraphics[scale=0.15]{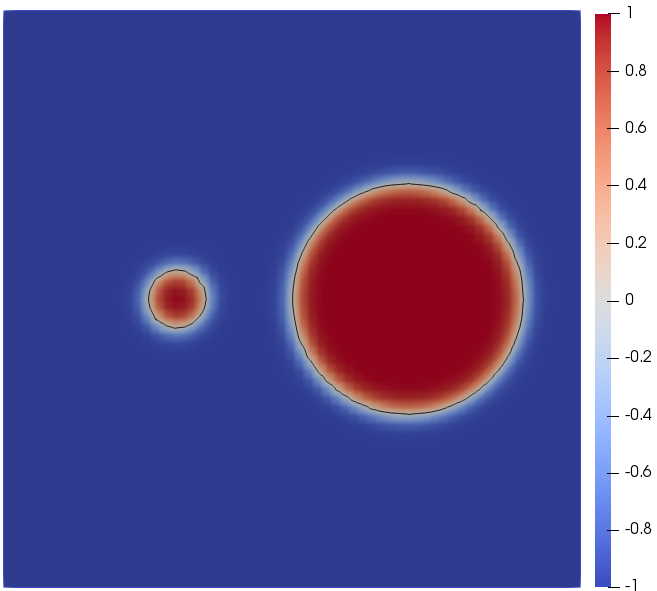}\tabularnewline
\end{tabular}
\par\end{centering}
\caption{Initial initial profile of the phase field for $\epsilon=0.1,0.05,0.025,0.0125$.\label{fig:Initphi-1}}
\end{figure}

We perform the experiments for this case as previous simulation. Fig.
\ref{fig:SnapsB} displays the phase field obtained with the diffuse-interface
model at the final time $T=2.5$ for various $\epsilon$. From this
figure, one can see that the two bubbles eventually coalesces into
a big bubble under the influence of surface tension for \textbf{Cases
I} and \textbf{III}, while for \textbf{Case II}, two bubbles are still
merging under the influence of surface tension. From these figure,
one can see that the diffuse-interface models still converge towards
their limits, even beyond the topological change which are not covered
by the presented theory. 

\begin{figure}
\begin{centering}
\subfloat[\textbf{Case I}: $M\left(\varphi\right)=m_{0}$ .]{\begin{centering}
\begin{tabular}{cccc}
\includegraphics[scale=0.15]{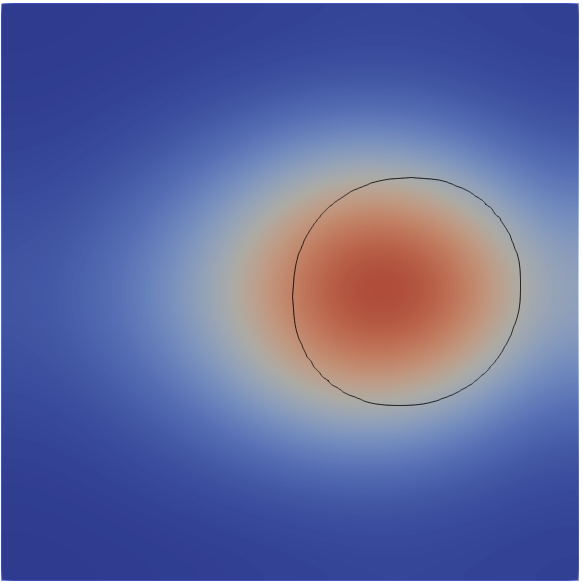} & \includegraphics[scale=0.15]{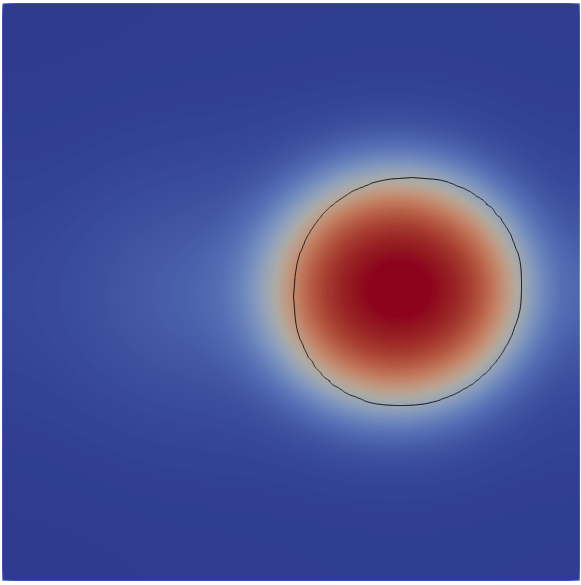} & \includegraphics[scale=0.15]{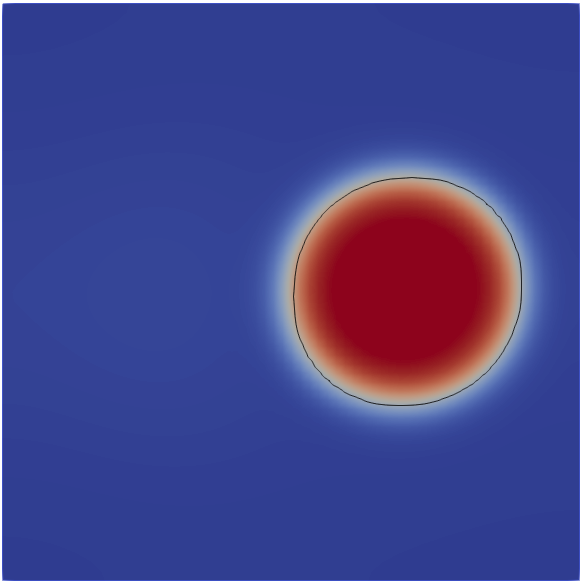} & \includegraphics[scale=0.15]{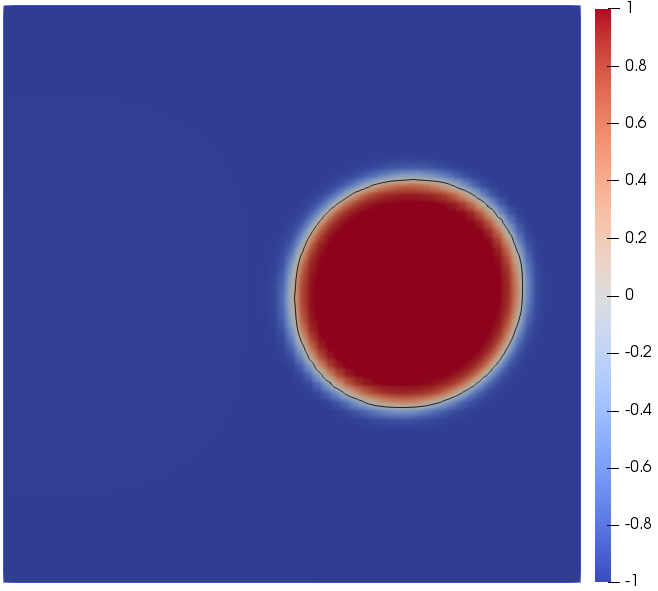}\tabularnewline
\end{tabular}
\par\end{centering}
}
\par\end{centering}
\begin{centering}
\subfloat[\textbf{Case II}: $M\left(\varphi\right)=\epsilon m_{0}$.]{\begin{centering}
\begin{tabular}{cccc}
\includegraphics[scale=0.15]{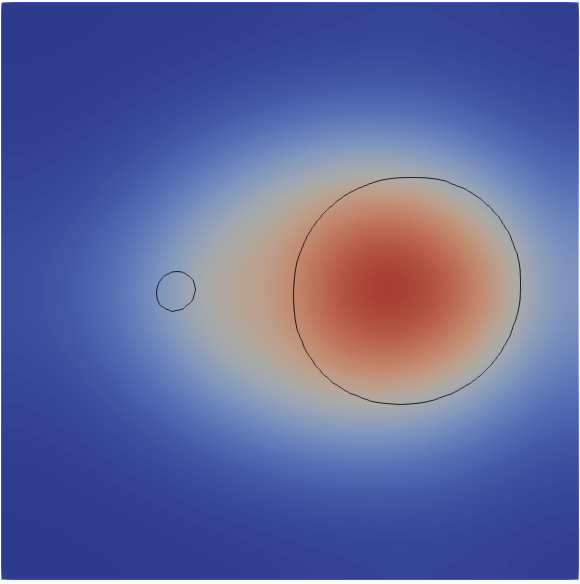} & \includegraphics[scale=0.15]{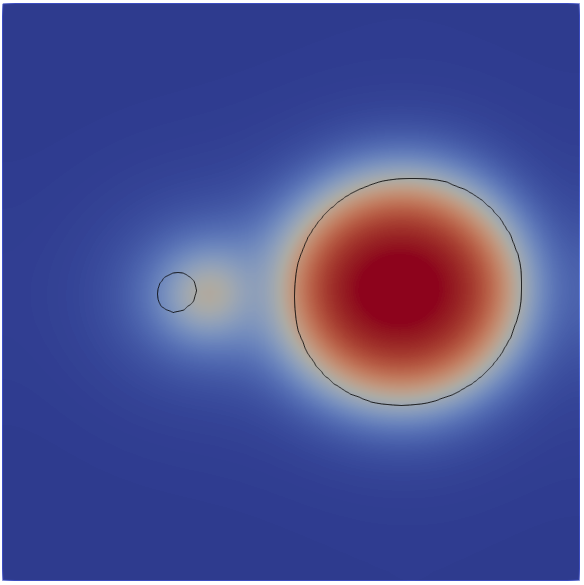} & \includegraphics[scale=0.15]{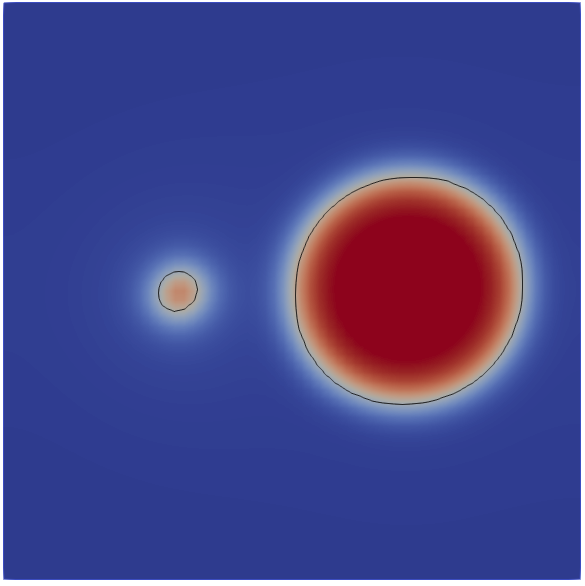} & \includegraphics[scale=0.15]{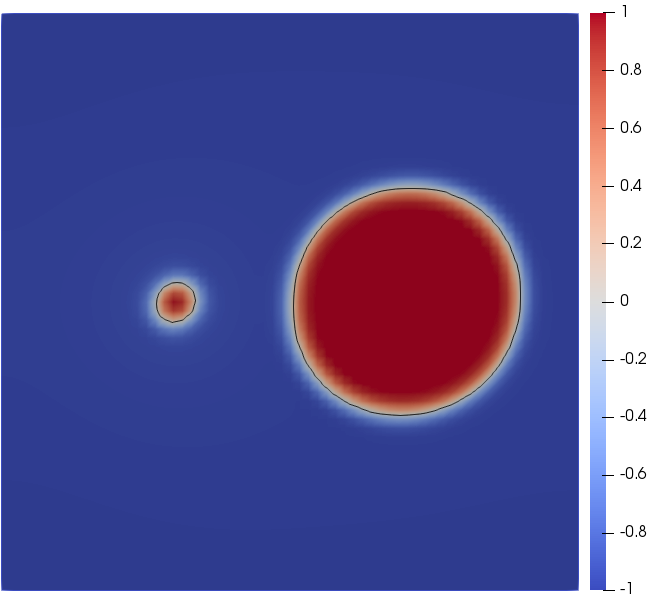}\tabularnewline
\end{tabular}
\par\end{centering}
}
\par\end{centering}
\begin{centering}
\subfloat[\textbf{Case III}: $M\left(\varphi\right)=m_{0}\left(1+\varphi\right)_{+}^{2}$.]{\begin{centering}
\begin{tabular}{cccc}
\includegraphics[scale=0.15]{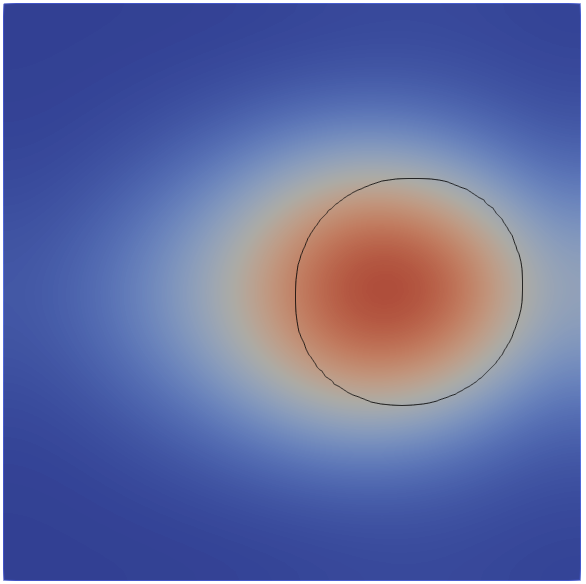} & \includegraphics[scale=0.15]{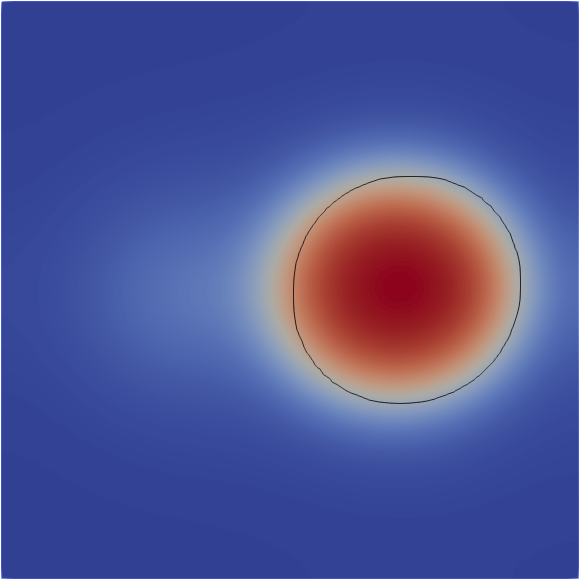} & \includegraphics[scale=0.15]{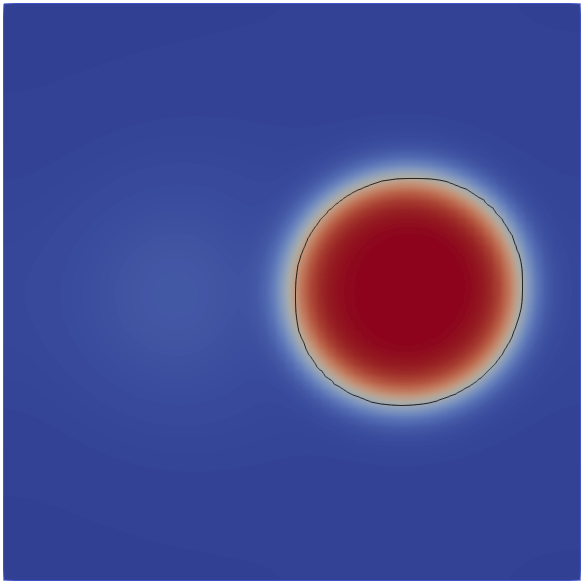} & \includegraphics[scale=0.15]{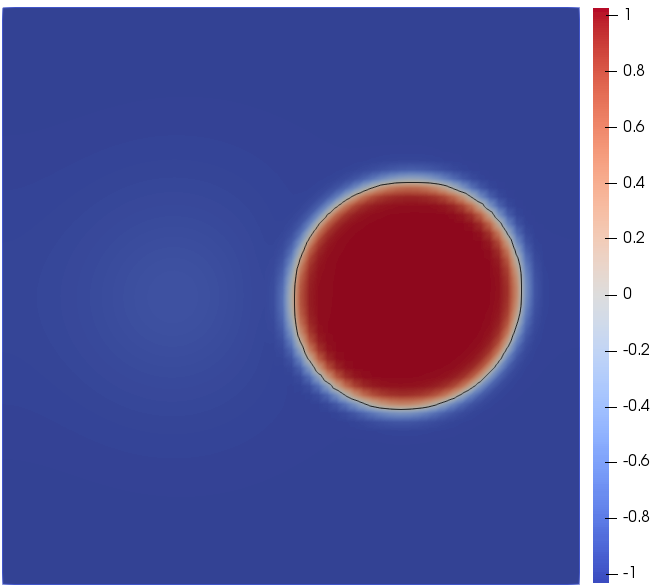}\tabularnewline
\end{tabular}
\par\end{centering}
}
\par\end{centering}
\caption{\label{fig:SnapsB}Profile of the phase fields for $\epsilon=0.1,0.05,0.025,0.0125$.(From
left to right.)}
\end{figure}

\section{Conclusions and remarks}

In this paper, we propose and analyze a diffuse interface model for
inductionless MHD fluids which couples a convective Cahn-Hilliard
equation for the evolution of the interface, the Navier\textendash Stokes
system for fluid flow and the possion equation for electrostatics.
The model is derived from Onsager's variational principle and conservation
laws systematically. Then we perform formally matched asymptotic expansions
and develop several sharp interface limits for the diffuse interface
model with different mobilitiess. Numerical results illustrate the
convergence. Our analysis and numerical studies will be helpful in
the real applications using the diffuse interface model for two-phase
inductionless MHD fluids. 

In upcoming works, we will investigate the model with large density
ratio and the induced magnetic field. We also plan to study the highly
efficient and energy stable schemes numerical methods for the corresponding
models.

\section*{Acknowledgments}
The author would like to thank Prof. Qi Wang and Prof. Tao Lin for their discussions and helpful suggestions.

\appendix

\section{CHiMHD model with moving contact lines\label{sec:CHiMHDMCL}}

Moving contact line is a challenging problem in fluid dynamics. Here
we give a sketch of an extension for the model (\ref{model:chimhd})
with moving contact lines. In the present case, the fluids is in contact
with solid surfaces at $\Sigma$. Let $\gamma_{fs}\left(\phi\right)$
is the solid-fluid interfacial energy density (up to a constant),
we add another term to the interfacial energy at the fluid-solid interface
\citep{Eck2009,Nochetto2014},

\[
E_{fs}=\int_{\Sigma}\gamma_{fs}(\varphi).
\]

\noindent Then the total energy in this case takes the form,
\[
E_{mcl}=E_{\varphi}+E_{\boldsymbol{u}}+E_{fs}.
\]

\noindent The variational derivative of the energy $E_{mcl}$ with
respect $\phi$ gives,

\noindent 
\[
\delta E_{mcl}=\int_{\Omega}\mu\delta\varphi+\int_{\Sigma}L\delta\varphi,
\]

\noindent with $L=\gamma_{fs}^{\prime}(\varphi)+\gamma_{0}\delta\frac{\partial\varphi}{\partial n}$
is a \textquotedblleft chemical potential\textquotedblright{} on the
boundary. For the boundary term, we introduce the material derivative
at the boundary $\dot{\varphi}=\partial_{t}\varphi+\boldsymbol{u}_{\tau}\partial_{\tau}\varphi$,$\boldsymbol{u}_{\tau}$
is the tangential fluid velocity along the boundary tangential direction
$\boldsymbol{\tau}$ and $\nabla_{\boldsymbol{\tau}}=\nabla-(\boldsymbol{n}\cdot\nabla)\boldsymbol{n}$
is the gradient along $\boldsymbol{\tau}$. With the similar arguments,
one can get expression for the forces $\boldsymbol{F}$. The the dissipation
function needs minor modifications,
\begin{equation}
\Phi\left(\mathbf{J},\mathbf{J}\right)=\int_{\Omega}\frac{\left|\boldsymbol{J}_{\varphi}\right|^{2}}{2M(\varphi)}+\int_{\Omega}\frac{|\boldsymbol{J}|^{2}}{2\sigma(\varphi)}+\int_{\Omega}\frac{|\mathbf{S}|^{2}}{2\eta(\varphi)}+\int_{\partial\Omega}\frac{\dot{\varphi}^{2}}{2\alpha^{-1}}+\int_{\partial\Omega}\frac{\beta}{2}\left|\boldsymbol{u}_{\tau}\right|^{2},\label{eq:dispmcl}
\end{equation}

\noindent where $\mathbf{J}=\left(\mathbf{S},\boldsymbol{J}_{\varphi},\boldsymbol{J},\dot{\varphi},\boldsymbol{u}_{\tau}\right)$.
Invoking with the Onsager's variational principle, we find the boundary
conditions for $\dot{\varphi}$ and $\boldsymbol{u}_{\tau}$ $\text{on }\Sigma$,
\begin{align}
\alpha\dot{\varphi} & =-L(\varphi)\label{eq:gnbc1}\\
\beta\boldsymbol{u}_{\tau} & =-\eta(\varphi)(\mathbf{S}\boldsymbol{n})_{\tau}+L(\varphi)\partial_{\tau}\varphi.\label{eq:gnbc2}
\end{align}

\noindent The boundary condition for the tangential velocity is known
as the generalized Navier boundary condition (GNBC) \citep{Xu2018,Wang2007}.
Thus, we obtain the model with moving contact lines, adding the GNBC
to the model (\ref{model:chimhd}). 

\section{The numerical scheme\label{sec:Scheme}}

\noindent Let $\mathcal{T}_{h}$ be a quasi-uniform and shape-regular
tetrahedral mesh of $\Omega$. As usual, we introduce the local mesh
size $h_{K}=\mathrm{diam}\left(K\right)$ and the global mesh size
$h:=\underset{K\in\mathcal{T}_{h}}{\max}h_{K}$. For any integer $k\geq0,$
let $P_{k}(K)$ be the space of polynomials of degree $k$ on element
$K$ and define $\boldsymbol{P}_{k}(K)=P_{k}(K)^{3}$. We employ the
Mini-element to approximate the velocity and pressure
\[
\boldsymbol{V}_{h}=\boldsymbol{P}_{1,h}^{b}\cap\boldsymbol{H}_{0}^{1}(\Omega),\quad Q_{h}=\left\{ q_{h}\in H^1(\Omega)\cap L_{0}^{2}(\Omega):\left.q_{h}\right|_{K}\in P_{1}(K),\quad\forall K\in\mathcal{T}_{h}\right\} 
\]
where $P_{1,h}^{b}=\left\{ v_{h}\in C^{0}(\Omega):\left.v_{h}\right|_{K}\in P_{1}(K)\oplus{\rm span}\{\hat{b}\},\;\forall K\subset T_{h}\right\} $,
$\hat{b}$ is a bubble function. We choose the lowest-order Raviart-Thomas
element space given by

\[
\boldsymbol{D}_{h}=\left\{ \boldsymbol{K}_{h}\in\boldsymbol{H}_{0}(\mathrm{div},\Omega):\left.\boldsymbol{K}_{h}\right|_{K}\in\boldsymbol{P}_{0}(K)+\boldsymbol{x}P_{0}(K),\quad\forall K\in\mathcal{T}_{h}\right\} ,
\]

\noindent combined with the discontinuous and piecewise constant finite
element space
\[
S_{h}=\left\{ \psi_{h}\in L_{0}^{2}(\Omega):\left.\psi_{h}\right|_{K}\in P_{0}(K),\quad\forall K\in\mathcal{T}_{h}\right\} .
\]

\noindent The phase field $\phi$ and chemical potential $\mu$ are
discretized by first order Lagrange finite element space $\left(X_{h},X_{h}\right)$,
\[
X_{h}=\left\{ \chi_{h}\in H^{1}(\Omega):\left.\chi_{h}\right|_{K}\in P_{1}(K),\quad\forall K\in\mathcal{T}_{h}\right\} .
\]
 
Let $\left\{ t_{n}=n\tau:\quad n=0,1,\cdots,N\right\} ,\tau=T/N,$
be an equidistant partition of the time interval $[0,T].$ For any
time dependent function $\omega\left(x,t\right)$, the full approximation
to $\omega\left(x,t_{n}\right)$ will be denoted by $\omega_{h}^{n}$.
A fully discrete finite element scheme for problem (\ref{model:chimhd})
reads as follows: Given the initial datum $\boldsymbol{u}^{0}$ and
$\phi^{0}$, we compute $\left(\boldsymbol{u}_{h}^{n+1},p_{h}^{n+1},\boldsymbol{J}_{h}^{n+1},\varphi_{h}^{n+1},\phi_{h}^{n+1},\mu_{h}^{n+1}\right)$,
$n=0,1,\cdots,N-1$, by the following three steps:

\textbf{Step 1:} Find $\left(\varphi_{h}^{n+1},\mu_{h}^{n+1}\right)\in X_{h}\times X_{h}$
such that for any $\left(\psi_{h},\chi_{h}\right)\in X_{h}\times X_{h}$
\begin{equation}
\begin{cases}
\left(\delta_{t}\varphi_{h}^{n+1},\psi_{h}\right)+\left(M(\varphi_{h}^{n})\nabla\mu_{h}^{n+1},\nabla\psi_{h}\right)+\tau\left(\varphi_{h}^{n}\nabla\mu_{h}^{n+1},\varphi_{h}^{n}\nabla\psi_{h}\right) & =\left(\varphi_{h}^{n}\boldsymbol{u}_{h}^{n},\nabla\psi_{h}\right),\\
\lambda\varepsilon\left(\nabla\varphi_{h}^{n+1},\nabla\chi_{h}\right)+\left(\frac{\lambda S}{\epsilon}(\varphi_{h}^{n+1}-\varphi_{h}^{n}),\chi_{h}\right)\\
+\left(\frac{\lambda}{\varepsilon}f(\varphi^{n}),\chi_{h}\right)-\left(\mu_{h}^{n+1},\chi_{h}\right) & =0,
\end{cases}\label{eq:CHt}
\end{equation}

\textbf{Step 2:} Find $\left(\boldsymbol{J}_{h}^{n+1},\varphi_{h}^{n+1}\right)\in\boldsymbol{D}_{h}\times S_{h}$
such that for any $\left(\boldsymbol{K}_{h},\theta_{h}\right)\in\boldsymbol{D}_{h}\times S_{h}$
\begin{equation}
\begin{cases}
\left(\sigma\left(\varphi_{h}^{n+1}\right)^{-1}\boldsymbol{J}_{h}^{n+1},\boldsymbol{K}_{h}\right)+\tau\left(\boldsymbol{J}_{h}^{n+1}\times\boldsymbol{B},\boldsymbol{K}_{h}^{n+1}\times\boldsymbol{B}\right)\\
-\left(\phi_{h}^{n+1},\mathrm{div}\boldsymbol{K}_{h}\right)-\tau\left(\varphi_{h}^{n+1}\nabla\mu_{h}^{n+1}\times\boldsymbol{B},\boldsymbol{K}_{h}^{n+1}\right) & =\left(\boldsymbol{u}_{h}^{n}\times\boldsymbol{B},\boldsymbol{K}_{h}^{n+1}\right),\\
\left(\mathrm{div}\boldsymbol{J}_{h}^{n+1},\theta_{h}\right) & =0,
\end{cases}\label{eq:Jphit}
\end{equation}

\textbf{Step 3:} Find $\left(\boldsymbol{u}^{n+1},p^{n+1}\right)\in\boldsymbol{V}_{h}\times Q_{h}$
such that $\left(\boldsymbol{v}_{h},q_{h}\right)\in\boldsymbol{V}_{h}\times Q_{h}$
\begin{equation}
\begin{cases}
(\delta_{t}\boldsymbol{u}_{h}^{n+1},\boldsymbol{v}_{h})+\mathcal{O}(\boldsymbol{u}_{h}^{n},\boldsymbol{u}_{h}^{n+1},\boldsymbol{v}_{h})+2\left(\eta\left(\phi^{n+1}\right)D\left(\boldsymbol{u}_{h}^{n+1}\right),D\left(\boldsymbol{v}_{h}^{n+1}\right)\right)\\
\qquad-\left(p_{h}^{n+1},{\rm div}\boldsymbol{v}_{h}\right)+\left(\varphi_{h}^{n+1}\nabla\mu_{h}^{n+1},\boldsymbol{v}_{h}\right)-\left(\boldsymbol{J}_{h}^{n+1}\times\boldsymbol{B},\boldsymbol{v}_{h}\right) & =\boldsymbol{0},\\
\left(\mathrm{div}\boldsymbol{u}^{n+1},q_{h}\right) & =0,
\end{cases}\label{eq:NSt}
\end{equation}

with $\mathcal{O}(\boldsymbol{u},\boldsymbol{v},\boldsymbol{w})\coloneqq 1/2 \left[(\boldsymbol{u}\nabla\cdot\boldsymbol{v},\boldsymbol{w})-(\boldsymbol{u}\nabla\cdot\boldsymbol{w},\boldsymbol{v})\right]$ for $\boldsymbol{u},\boldsymbol{v},\boldsymbol{w}\in\boldsymbol{V}$.

\noindent The scheme (\ref{eq:CHt})-(\ref{eq:NSt}) can be proved
to be unconditional energy stable and charge-conservative. The finite
element method is implemented on the finite element software FreeFEM
developed by Hecht et al. \citep{Hecht2012}. In all the simulations
in this paper, we choose the meshsize $h=1/64$ and timestep $\tau=0.01$.

\section*{References}

\bibliographystyle{elsarticle-num}
\bibliography{mybib}

\end{document}